\documentclass[11pt]{amsart}
\usepackage{fullpage}
\usepackage[utf8]{inputenc}
\usepackage{amssymb, amsmath, wasysym, mathrsfs, mathtools, color, stmaryrd,enumerate}
\usepackage[mathscr]{euscript}
\theoremstyle{plain}
\usepackage{url}
\usepackage{csquotes}

\usepackage{tikz}
\usetikzlibrary{positioning}

\usepackage{tikz-cd}
\usepackage{tkz-euclide}
\usepackage[all]{xy}
\tikzcdset{row sep/my size/.initial=5em}
\tikzcdset{column sep/my size/.initial=10em}

\definecolor{hot}{RGB}{65,105,225}
\usepackage[pagebackref=true,colorlinks=true, linkcolor=hot ,  citecolor=hot, urlcolor=hot]{hyperref}

\newtheorem{theorem}{Theorem}[section]

\theoremstyle{definition}
\newtheorem{defi}{Definition}[section]

\newtheorem{const}{Construction}[section]

\theoremstyle{remark}
\newtheorem{remark}[theorem]{Remark}

\numberwithin{equation}{section}

\theoremstyle{plain}

\newtheorem{thm}{Theorem}[section]

\newtheorem{prop}{Proposition}[section]

\newtheorem{coro}{Corollary}[section]

\newtheorem*{quest}{Question}
\theoremstyle{remark}
\newtheorem{rem}{Remark}[section]

\newtheorem*{ackn}{Acknowledgements}

\newtheorem*{thm*}{Theorem}

\newcommand{\Sp}{\operatorname{Spec}}

\newcommand{\dg}{\operatorname{\mathbf{dgArt}\!}{}}

\newcommand{\ens}{\operatorname{\mathbf{\mathscr{S}}\!}{}}

\newcommand{\Hom}{\operatorname{Hom\!}{}}

\newcommand{\Mo}{\operatorname{Mod\!}{}}
\newcommand{\Mod}{\operatorname{\mathbf{Mod}\!}{}}
\newcommand{\rep}{\operatorname{\mathbf{Rep}\!}{}}
\newcommand{\Rep}{\operatorname{Rep\!}{}}

\newcommand{\cdga}{\operatorname{cdga\!}{}}
\newcommand{\cdg}{\operatorname{\mathbf{cdga}\!}{}}

\newcommand{\Calg}{\operatorname{\textbf{CAlg}\!}{}}
\newcommand{\Fun}{\operatorname{\mathbf{Fun}\!}{}}
\newcommand{\Li}{\operatorname{Lie\!}{}}
\newcommand{\Lie}{\operatorname{\mathbf{Lie}\!}{}}
\newcommand{\Cn}{\operatorname{Cn\!}{}}

\newcommand{\Ho}{\operatorname{\mathbf{Ho}\!}{}}

\newcommand{\Map}{\operatorname{Map\!}{}}

\newcommand{\se}{\operatorname{\mathbf{Set}\!}{}}

\newcommand{\fin}{\operatorname{\mathscr{F}\mathrm{in}_*\!}{}}
\newcommand{\cat}{\operatorname{\mathscr{C}\mathrm{at}_{\infty}\!}{}}
\newcommand{\qc}{\operatorname{\mathbf{QCoh}\!}{}}
\newcommand{\qqc}{\operatorname{Coh\!}{}}
\newcommand{\Art}{\operatorname{\mathbf{Art}\!}{}}
\newcommand{\fun}{\operatorname{Fun\!}{}}

\begin{document}

\title{Formal moduli problems with cohomological constraints}

\author{An-Khuong Doan }

\address{An-Khuong DOAN, Department of Mathematics, KU Leuven, Celestijnenlaan 200B, 3001 Leuven, Belgium.}\email{an-khuong.doan@kuleuven.be}

\address{IMJ-PRG, UMR 7586, Sorbonne Université,  Case 247, 4 place Jussieu, 75252 Paris Cedex 05, France}
\thanks{ }

\address{CNRS-CMLS, UMR 7640, \'{E}cole polytechnique,  91128 Palaiseau Cedex, France}


\date{August 26, 2024.}

\dedicatory{ }


\begin{abstract} We aim to generalize Lurie-Pridham's $\infty$-correspondence between formal moduli problems and  differential graded Lie algebras to the context where cohomological conditions are imposed. Specifically, a natural equivalence between formal moduli problems with cohomological constraints  and derived cohomology jump functors is provided, thereby answering affirmatively a question posed by N. Budur and B. Wang in \cite{1}.
\end{abstract}

\maketitle

\setcounter{tocdepth}{1}

\numberwithin{equation}{section}

\tableofcontents

\section{Introduction}  In classical deformation theory, there is a well-known principle which says that ``any reasonable deformation problem  is controlled by a differential graded Lie algebra and two quasi-isomorphic differential graded Lie algebras define the same deformation problem" whose first trace can be found in the letter written by P. Deligne to J. Milson in 1986 (cf. \cite{2}). This might be thought as an experimental observation from the condensation of almost three decades of works on concrete deformation problems among which we must mention the foundational one of Kodaira-Spencer on deformations of complex structures, in 1958 (cf. \cite{5}). In this case, deformations of a complex compact manifold $X_0$ are controlled by the Dolbeault complex with values in the holomorphic tangent bundle $\mathcal{T}_{X_0}$
\begin{equation} \label{e.1} \Gamma(X_0, \mathcal{A}^{0,0}(\mathcal{T}_{X_0}))\overset{\bar{\partial}}{\rightarrow}\Gamma(X_0, \mathcal{A}^{0,1}(\mathcal{T}_{X_0}))\overset{\bar{\partial}}{\rightarrow}\Gamma(X_0, \mathcal{A}^{0,2}(\mathcal{T}_{X_0}))\overset{\bar{\partial}}{\rightarrow}\cdots
\end{equation} with the Lie bracket defined by 
$$[\phi d\bar{z}_I, \psi  d\bar{z}_J]=[\phi,\psi]' d\bar{z}_I \wedge\bar{z}_J$$ where $\phi,\psi \in \mathcal{A}^{0,0}(\mathcal{T}_{X_0})$ are vector fields on $X_0$, $[-,-]'$ is the usual Lie bracket of vector fields, $I,J\subset \lbrace 1,\ldots, n \rbrace$ and $z_1,\ldots, z_n$ are local holomorphic coordinates (cf. \cite{5, Man04, Man22}). The same phenomenon happens to deformations of a holomorphic vector bundle $E$ defined over a complex compact manifold where the controlling differential graded Lie algebra is obtained by replacing the holomorphic tangent bundle $\mathcal{T}_{X_0}$ in \ref{e.1} by the endomorphism bundle $\mathrm{End}(E)$ (cf. \cite{Fuk03, Man22}). These two examples are living evidence for the aforementioned credo and then  naturally suggest a mathematically rigorous formulation of it. This had not been done for a long time until the appearance of Lurie's work (cf. \cite{6}) and independently that of Pridham (cf. \cite{11}), both from around 2010 in the framework of derived algebraic geometry which give an end to the question ``In which context is Deligne's principle a theorem?". Namely, with appropriate axiomatizations, they prove that there is an $\infty$-equivalence between the category of formal moduli problems (a natural derived version of classical deformation problems) and that of differential graded Lie algebras (see Theorem \ref{t2.1} below), which is  revolutionary in derived deformation theory. Moreover, it generalized Schlessinger's work (cf. \cite{9}) on functors of artinian rings to its extremity. For a detailed history about this revolution,  we recommend the beautiful seminar paper \cite{12} of B. To\"{e}n to the curious reader.

N. Budur and B. Wang defined in \cite{1} a new kind of deformation problems: the ones with cohomological constraints. Concretely speaking, if classical deformation problems are viewed as ``good" functors of artinian rings, then the ones with cohomological constraints are their ``good" sub-functors. To do this, given a classical deformation problem whose controlling differential graded Lie algebra is $\mathfrak{g}_*$, and given any $\mathfrak{g}_*$-module $M_*$, by means of a generalized version of Fitting ideals for chain complexes of finitely generated modules over noetherian rings, they defined canonically sub-functors of the Maurer-Cartan functor associated to $\mathfrak{g}_* $, to which they referred  as ``cohomology jump functors" attached to the pair $(\mathfrak{g}_*,M_*)$ (cf. Definition \ref{d4.3}). Moreover, they proposed a cohomological constraint version of Deligne's principle that in reality, any deformation problem with cohomological constraints over a field of characteristic zero should arise in this way and that two quasi-isomorphic pairs define the same deformation problem with cohomological constraints (cf. \cite[Theorem 3.16]{1}). The typical example is deformations of a holomorphic vector bundle $E$ with cohomological constraints (cf. \cite[Theorem 6.4]{1}). They are controlled by the pair $(\mathfrak{g}_*^E, M^E_*)$ where $\mathfrak{g}_*^E$ and $M^E_*$ are the Dolbeault complexes with values in $\mathrm{End}(E)$ and in $E$, respectively. The $\mathfrak{g}_*^E$-module structure of $M^E_*$ comes from the natural map $\mathrm{End}(E)\otimes E \rightarrow E$. Furthermore, deformations with cohomology constraints have been generalized to the $L_\infty$ setting - a much more flexible framework for practical purposes where differential graded pairs (dgLa, dg-module) are replaced by $L_\infty$ pairs ($L_\infty$ algebra, $L_\infty$ module) (cf. \cite{BR18}).
 Inspired by Lurie-Pridham's realization of Deligne's principle, Budur and Wang posed the following natural question.

\begin{quest} Does there exist a similar version of Lurie-Pridham's $\infty$-equivalence theorem between deformation problems with cohomological constraints and pairs of differential graded Lie algebras and their modules?
\end{quest}
Our objective in this paper is to give an affirmative answer to this question. More precisely, we study quasi-coherent sheaves on formal moduli problems and link them with representations of associated differential graded Lie algebras. These objects are well-developed by J. Lurie (cf. \cite{6}).  It turns out to be that imposing cohomological constraints on such sheaves gives rise to usual cohomology jump functors thereby obtaining an $\infty$-formulation of Budur-Wang's principle.

Let us now outline the organization of this article. The next three sections are dedicated to a quick review of representations of differential graded Lie algebras (or modules over differential graded Lie algebras, according to the terminology of N. Budur and B. Wang) and quasi-coherent sheaves on formal moduli problems. In $\S\ref{s4}$, we recall the construction of cohomology jump functors of which we propose a derived version. In $\S\ref{s5}$, we reply to the question of N. Budur and B. Wang by providing a version of Lurie-Pridham's theorem with cohomological constraints (cf. Theorem \ref{t5.1} and Corollary \ref{t5.1}). Finally, in $\S\ref{s6}$, we study some interesting classes of formal moduli problems with cohomological constraints. 
\newline
\newline
\textbf{Conventions and notations}: \begin{enumerate}
\item[$\bullet$] A field of characteristic $0$ will be always denoted by $k$.
\item[$\bullet$] dgLa is the abbreviation of differential graded Lie $k$-algebra while cdga means commutative differential graded augmented $k$-algebra.
\item[$\bullet$] $\Mo_k$ is the category of chain complexes of $k$-modules and $\Mod_k$ is the corresponding $\infty$-category.
\item[$\bullet$] $\Li_k$ is the category of differential graded Lie $k$-algebras and $\Lie_k$ is the corresponding $\infty$-category.
\item[$\bullet$] $\cdga_k$ is the category of commutative differential graded augmented $k$-algebras and $\cdg_k$ is the corresponding $\infty$-category.
\item[$\bullet$] For any $A_* \in \cdg_k$ (respectively, $\mathfrak{g}_* \in \Lie_k$), $\Mod_{A_*}$ (respectively, $\Mod_{\mathfrak{g}_*}$) denotes the corresponding $\infty$-category of $A_*$-modules (respectively, $\mathfrak{g}_*$-modules).
\item[$\bullet$] $\dg_k$ denotes the full sub-category of $\cdg_k$ consisting of commutative differential graded artinian algebras cohomologically concentrated in non-positive degrees.
\item[$\bullet$] $\Art_k$ is the category of artinian rings with residue field $k$.
\item[$\bullet$] $\ens$ is the category of simplicial sets.
\item[$\bullet$] $\mathcal{FMP}$  is the homotopy category of formal moduli problems. 
\item[$\bullet$] For a given formal moduli problem $X$, $\qc(X)$ denotes the $\infty$-category of quasi-coherent sheaves on $X$. 
\end{enumerate}

\begin{ackn}  The author would like to profoundly thank J. Lurie for explaining some notions in \cite{6} on which the paper relies heavily. A warm thank is sent to N. Budur for several precious comments and interesting questions which helped to improve not only the quality but also the readability of the paper.  Finally, the author is specially thankful to the referee whose dedicated work led to a great amelioration of this paper.

The work is supported by the Institut de Mathématiques de Jussieu-Paris Rive Gauche, the Centre de Math\'{e}matiques Laurent Schwartz and the Grant G0B3123N of N. Budur from FWO.
\end{ackn}
\section{Preliminaries on higher algebra} In this section, we recall some elementary notions in higher algebra, which shall be repeatedly used throughout the rest of the paper. For more details, the interested reader is referred to \cite[Chapter 2]{7}. We start with the notion of symmetric monoidal categories. 
\begin{defi} A \textit{symmetric monoidal category} is a monoidal category $(\mathscr{C},\otimes, I)$ such that, for every pair $A, B$ of objects in $\mathscr{C}$, there is an isomorphism  $s_{AB}:A\otimes B\to B\otimes A$ that is natural in both $A$ and $B$ and such that the following diagrams commute:

\begin{enumerate}
\item[$\bullet$] The \textit{unit coherence}
     
\begin{center}
\tikzset{every picture/.style={line width=0.75pt}} 

\begin{tikzpicture}[x=0.6pt,y=0.6pt,yscale=-1,xscale=1]

\draw    (139,120) -- (237.59,218.59) ;
\draw [shift={(239,220)}, rotate = 225] [color={rgb, 255:red, 0; green, 0; blue, 0 }  ][line width=0.75]    (10.93,-3.29) .. controls (6.95,-1.4) and (3.31,-0.3) .. (0,0) .. controls (3.31,0.3) and (6.95,1.4) .. (10.93,3.29)   ;
\draw    (160,110) -- (321,110) ;
\draw [shift={(323,110)}, rotate = 180] [color={rgb, 255:red, 0; green, 0; blue, 0 }  ][line width=0.75]    (10.93,-3.29) .. controls (6.95,-1.4) and (3.31,-0.3) .. (0,0) .. controls (3.31,0.3) and (6.95,1.4) .. (10.93,3.29)   ;
\draw    (341,121) -- (252.35,218.52) ;
\draw [shift={(251,220)}, rotate = 312.27] [color={rgb, 255:red, 0; green, 0; blue, 0 }  ][line width=0.75]    (10.93,-3.29) .. controls (6.95,-1.4) and (3.31,-0.3) .. (0,0) .. controls (3.31,0.3) and (6.95,1.4) .. (10.93,3.29)   ;

\draw (108,102.15) node [anchor=north west][inner sep=0.75pt]    {$A\otimes I$};
\draw (329.75,102.15) node [anchor=north west][inner sep=0.75pt]    {$I\otimes A$};
\draw (233.75,223.15) node [anchor=north west][inner sep=0.75pt]    {$A$};
\draw (233.75,92.15) node [anchor=north west][inner sep=0.75pt]    {$s_{AI}$};
\draw (160.75,166.15) node [anchor=north west][inner sep=0.75pt]    {$r_{A}$};
\draw (304.75,166.15) node [anchor=north west][inner sep=0.75pt]    {$l_{A}$};

\end{tikzpicture}

\end{center}

\item[$\bullet$] The \textit{associativity coherence}

\begin{center}

\tikzset{every picture/.style={line width=0.75pt}} 

\begin{tikzpicture}[x=0.6pt,y=0.6pt,yscale=-1,xscale=1]

\draw    (191,40) -- (393,40) ;
\draw [shift={(395,40)}, rotate = 180] [color={rgb, 255:red, 0; green, 0; blue, 0 }  ][line width=0.75]    (10.93,-3.29) .. controls (6.95,-1.4) and (3.31,-0.3) .. (0,0) .. controls (3.31,0.3) and (6.95,1.4) .. (10.93,3.29)   ;
\draw    (191,220) -- (393,220) ;
\draw [shift={(395,220)}, rotate = 180] [color={rgb, 255:red, 0; green, 0; blue, 0 }  ][line width=0.75]    (10.93,-3.29) .. controls (6.95,-1.4) and (3.31,-0.3) .. (0,0) .. controls (3.31,0.3) and (6.95,1.4) .. (10.93,3.29)   ;
\draw    (140,50) -- (140,114) ;
\draw [shift={(140,116)}, rotate = 270] [color={rgb, 255:red, 0; green, 0; blue, 0 }  ][line width=0.75]    (10.93,-3.29) .. controls (6.95,-1.4) and (3.31,-0.3) .. (0,0) .. controls (3.31,0.3) and (6.95,1.4) .. (10.93,3.29)   ;
\draw    (140,140) -- (140,204) ;
\draw [shift={(140,206)}, rotate = 270] [color={rgb, 255:red, 0; green, 0; blue, 0 }  ][line width=0.75]    (10.93,-3.29) .. controls (6.95,-1.4) and (3.31,-0.3) .. (0,0) .. controls (3.31,0.3) and (6.95,1.4) .. (10.93,3.29)   ;
\draw    (456,50) -- (456,114) ;
\draw [shift={(456,116)}, rotate = 270] [color={rgb, 255:red, 0; green, 0; blue, 0 }  ][line width=0.75]    (10.93,-3.29) .. controls (6.95,-1.4) and (3.31,-0.3) .. (0,0) .. controls (3.31,0.3) and (6.95,1.4) .. (10.93,3.29)   ;
\draw    (456,141) -- (456,205) ;
\draw [shift={(456,207)}, rotate = 270] [color={rgb, 255:red, 0; green, 0; blue, 0 }  ][line width=0.75]    (10.93,-3.29) .. controls (6.95,-1.4) and (3.31,-0.3) .. (0,0) .. controls (3.31,0.3) and (6.95,1.4) .. (10.93,3.29)   ;

\draw (84.75,31.15) node [anchor=north west][inner sep=0.75pt]    {$( A\otimes B) \otimes C$};
\draw (398.75,31.15) node [anchor=north west][inner sep=0.75pt]    {$( B\otimes A) \otimes C$};
\draw (84.75,122.15) node [anchor=north west][inner sep=0.75pt]    {$A\otimes ( B\otimes C)$};
\draw (398.75,122.15) node [anchor=north west][inner sep=0.75pt]    {$B\otimes ( A\otimes C)$};
\draw (84.75,212.15) node [anchor=north west][inner sep=0.75pt]    {$( B\otimes C) \otimes A$};
\draw (398.75,212.15) node [anchor=north west][inner sep=0.75pt]    {$B\otimes ( C\otimes A)$};
\draw (267.75,19.15) node [anchor=north west][inner sep=0.75pt]    {$s_{AB} \otimes \mathrm{1}_{C}$};
\draw (88.75,72.15) node [anchor=north west][inner sep=0.75pt]    {$a_{ABC}$};
\draw (463.75,72.15) node [anchor=north west][inner sep=0.75pt]    {$a_{BAC}$};
\draw (73.75,160.15) node [anchor=north west][inner sep=0.75pt]    {$a_{A,B\otimes C}$};
\draw (463.75,160.15) node [anchor=north west][inner sep=0.75pt]    {$\mathrm{1}_{B} \otimes s_{AC}$};
\draw (273.75,223.15) node [anchor=north west][inner sep=0.75pt]    {$a_{BCA}$};

\end{tikzpicture}

\end{center}
\item[$\bullet$] The \textit{inverse law}
\begin{center}

\tikzset{every picture/.style={line width=0.75pt}} 

\begin{tikzpicture}[x=0.6pt,y=0.6pt,yscale=-1,xscale=1]

\draw    (161,220) -- (260.71,102.52) ;
\draw [shift={(262,101)}, rotate = 130.32] [color={rgb, 255:red, 0; green, 0; blue, 0 }  ][line width=0.75]    (10.93,-3.29) .. controls (6.95,-1.4) and (3.31,-0.3) .. (0,0) .. controls (3.31,0.3) and (6.95,1.4) .. (10.93,3.29)   ;
\draw    (270,102) -- (368.71,219.47) ;
\draw [shift={(370,221)}, rotate = 229.96] [color={rgb, 255:red, 0; green, 0; blue, 0 }  ][line width=0.75]    (10.93,-3.29) .. controls (6.95,-1.4) and (3.31,-0.3) .. (0,0) .. controls (3.31,0.3) and (6.95,1.4) .. (10.93,3.29)   ;
\draw    (182,230) -- (351,230) ;
\draw    (182,236) -- (351,236) ;

\draw (125.75,225.15) node [anchor=north west][inner sep=0.75pt]    {$A\otimes B$};
\draw (353.75,225.15) node [anchor=north west][inner sep=0.75pt]    {$A\otimes B$};
\draw (239.75,80.15) node [anchor=north west][inner sep=0.75pt]    {$B\otimes A$};
\draw (183.75,137.15) node [anchor=north west][inner sep=0.75pt]    {$s_{AB}$};
\draw (315.75,137.15) node [anchor=north west][inner sep=0.75pt]    {$s_{BA}$};

\end{tikzpicture}

\end{center}
\end{enumerate}
\end{defi}

\begin{const}\cite[Construction 2.0.0.1]{7} \label{maincons} Let $(\mathscr{C}, \otimes)$ be a symmetric monoidal category. We can build up a new category $\mathscr{C}^{\otimes}$ from $\mathscr{C}$ in the following way.
\begin{enumerate}
\item[(i)] An object of $\mathcal{C}^{\otimes}$ is a finite (maybe empty) sequence of objects of $\mathcal{C}$, denoted by $[C_1,\ldots, C_n]$.
\item[(ii)] A morphism from $[C_1,\ldots, C_n]$ to $[C_1',\ldots, C_m']$ consists of a subset $S \subseteq \lbrace 1,\ldots, n\rbrace$, a map of finite sets $\alpha: S \rightarrow \lbrace 1,\ldots, m\rbrace$ and a collection of morphisms $\lbrace f_j: \bigotimes _{\alpha(i)=j} C_i \rightarrow C_j'\rbrace_{1\leq j\leq m}$ in the category $\mathscr{C}$.
\item[(iii)] Given two morphisms $f: [C_1,\ldots, C_n] \rightarrow [C_1',\ldots, C_m']$ and $g: [C_1',\ldots, C_m'] \rightarrow [C_1'',\ldots, C_l''] $ in $\mathcal{C}^{\otimes}$, determining subsets $S \in \lbrace 1,\ldots, n\rbrace$ and $T \in \lbrace 1,\ldots, m\rbrace$ together with maps $\alpha: S \rightarrow \lbrace 1,\ldots, n\rbrace$ and $\beta: T \rightarrow \lbrace 1,\ldots, l\rbrace$. The composition $g\circ f $ is given by the subset $\beta^{-1}(T)\subseteq \lbrace 1,\ldots, n\rbrace$, the composite map $\beta \circ \alpha: U \rightarrow \lbrace 1,\ldots, l\rbrace$ and the collection of maps
$$\bigotimes_{\beta\circ \alpha (i)=k}C_i \cong \bigotimes_{\beta(j)=k}\bigotimes_{\alpha(i)=j}C_i \rightarrow \bigotimes _{\beta(j)=k}C_j' \rightarrow C_k''$$ for $1 \leq k \leq l$.
\end{enumerate}

\end{const}
In order to understand more profoundly the category $\mathscr{C}^{\otimes}$, we recall also the definition of Segal's category $\mathscr{F}\mathrm{in}_*$ of pointed finite sets. 
\begin{defi} \cite[Notation 2.0.0.2]{7} $\mathscr{F}\mathrm{in}_*$ is the category whose 
\begin{enumerate}
\item[(i)] objects are sets $\left< n\right>:= \lbrace * ,1,\ldots,n \rbrace$ for $n\in \mathbb{N}$.
\item[(ii)] morphisms are maps of finite sets $\alpha: \left<m \right> \rightarrow \left< n\right>$ such that $\alpha (*)=*$, where $\left< m \right>$ and $\left<n \right>$ are objects of $\mathscr{F}\mathrm{in}_*$.
\end{enumerate}
\end{defi} 

The following class of morphisms in $\mathscr{F}\mathrm{in}_*$ shall play a crucial role in describing axiomatically the symmetric monoidal structure on a given category. For each pair of nonzero natural numbers $(i,n)$ where $i\leq n$, we denote by $\rho^i :\left< n \right>  \rightarrow \left< 1 \right> $ the morphism defined by the rule
$$\rho^i(j)=\begin{cases}
1 & \text{ if } i=j \\
* & \text{ otherwise.} 
\end{cases}$$

For any symmetric monoidal category $\mathscr{C}$, we can consider the obvious forgetful functor from the category $\mathscr{C}^{\otimes}$ in Construction \ref{maincons} to the category $\mathscr{F}\mathrm{in}_*$ of pointed finite sets, i.e. the following functor
  \begin{align*}
p:  \mathscr{C}^{\otimes} & \rightarrow \mathscr{F}\mathrm{in}_*\\
[C_1,\ldots, C_n] &\mapsto \left< n \right>.\end{align*} 

This functor has two essential properties from which  the symmetric monoidal structure on $\mathscr{C}$ can be eventually recovered up to symmetric monoidal equivalence.
\begin{prop}\cite[Axioms (M1) and (M2), Page 167-168]{7} The following properties are fundamental:
\begin{enumerate} 
\item[(A1)] The functor $p$ is a cofibration of categories: for any object $C=[C_1,\ldots, C_n] \in \mathscr{C}^{\otimes}$ and every morphism $f:\left< n \right> \rightarrow \left< m \right> $ in  $\mathscr{F}\mathrm{in}_*$, there exists a morphism $\bar{f}: [C_1,\ldots, C_n] \rightarrow C'=[C_1',\ldots, C_m']$ in $\mathscr{C}^{\otimes}$, whose image under $p$ is exactly $f$. Furthermore, $\bar{f}$ is universal in the sense that the composition with $\bar{f}$ induces a bijective correspondence 
$$\;\;\;\;\;\;\;\;\;\;\;\; \Hom_{\mathscr{C}^{\otimes}}(C',[C_1'',\ldots, C_l'']) \rightarrow \Hom_{\mathscr{C}^{\otimes}}(C,[C_1'',\ldots, C_l''])\times _{\Hom_{\fin}(\left< n \right>,\left< l \right>)} \Hom_{\fin}(\left< m \right>,\left< l \right>)$$ for any object $C''=[C_1'',\ldots, C_l''] $ in $\mathscr{C}^{\otimes}$. Here $\bar f$ is chosen so that we have an isomorphism between $C'_j\simeq \otimes_{f(i)=j}C_i$ for $1\leq j \leq m$. Consequently, any morphism $f:\left< n \right> \rightarrow \left< m \right> $ in  $\mathscr{F}\mathrm{in}_*$ induces a functor between the fibers $\mathscr{C}^{\otimes}_{\left< n \right>} \rightarrow \mathscr{C}^{\otimes}_{\left< m \right>}$ unique up to canonical isomorphism.
\item[(A2)] for each $n \geq 1$, the class of morphisms $\lbrace \rho^i :\left< n \right>  \rightarrow \left< 1 \right> \rbrace_{1\leq i \leq n} $ induces an equivalence between the fiber $\mathscr{C}^{\otimes}_{\left< n \right>}$ of $p$ over  $\left< n \right>$ and  the $n$-fold product of copies of $\mathscr{C}$. 
\end{enumerate}
\end{prop}
We are going to see that the advantage of these two axioms is clearly revealed in Definition \ref{symoincat} when one has to deal with $\infty$-categories, after the following two remarks. 
\begin{rem} Recall that a \textit{simplicial set} is just a functor $X: \Delta^{op}\rightarrow \mathrm{Set}$ where $\Delta$ is the simplicial category whose opposite category is $\Delta^{op}$ . The \textit{opposite simplicial set} of $X$, denoted by $X^{op}$, is defined to be the functor $X\circ \mathrm{op}_{\Delta}$ where $\mathrm{op}_{\Delta}: \Delta \rightarrow \Delta^{op}$ is the functor that reverses the arrows in $\Delta$ and acts as the identity on objects.
\end{rem}
\begin{rem}\cite[Definition 2.4.2.1]{8} We say that a map $p: X \rightarrow S$ of simplicial sets is a \textit{Cartesian fibration} if the following conditions are satisfied:
\begin{enumerate}
\item[(1)] The map $p$ is an inner fibration: that is, it has the right lifting property with respect to all horn inclusions $\Lambda_i^n \subset \Delta^n$ for $0< i< n$.

\item[(2)] For every edge $f:x \rightarrow y$ of $S$ and every vertex $\tilde{y}$ of $X$ with $p(\tilde{y})=y$, there exists a $p$-Cartesian edge $\tilde{f}: \tilde{x}\rightarrow \tilde{y}$ with $p(\tilde{f})=f$. 
\end{enumerate}
We say that $p$ is a \textit{coCartesian} if the opposite map $p^{op}: X^{op}\rightarrow S^{op}$ is a Cartesian fibration.
\end{rem}
\begin{defi}\cite[Definition 2.0.0.7, Example 2.1.2.18]{7} \label{symoincat} Let $\mathscr{C}$ be a simplicial category.  A structure of \textit{symmetric monoidal} $\infty$-\textit{category} on $\mathscr{C}$ is a coCartesian fibration of simplical sets $$ p: \mathscr{C}^{\otimes}  \rightarrow \mathrm{N}(\mathscr{F}\mathrm{in}_*)$$ satisfying the property that for any $n\geq 1$, the class of maps $\lbrace \rho^i :\left< n \right>  \rightarrow \left< 1 \right> \rbrace_{1\leq i \leq n} $ induces functors 
$\rho^i :\mathscr{C}^{\otimes}_{\left< n \right> } \rightarrow \mathscr{C}^{\otimes}_{\left< 1 \right>} $ which determines equivalences $\mathscr{C}^{\otimes}_{\left< n \right> } \cong (\mathscr{C}^{\otimes}_{\left< 1 \right>})^n$ of $\infty$-categories. Here,   $\mathrm{N}(\mathscr{F}\mathrm{in}_*)$ is the nerve of $\mathscr{F}\mathrm{in}_*$, which is a simplicial set.
\end{defi}
We refer the interested reader once again to \cite[Chapter 2]{7} to see why this definition fits well into the $\infty$-categorial setting. 

Let $\mathscr{C}$ be a symmetric monoidal $\infty$-category (cf. \cite[Section 2.1.3.1]{7}). We refer to $\Calg(\mathscr{C})$ the $\infty$-category of commutative associative algebra objects of $\mathscr{C}$. We let $\Mod(\mathscr{C})^{\otimes}$ be as in \cite[Definition 3.3.3.8]{7}. More precisely, the objects of this category are given by tuples $(A,M_1, \ldots, M_n)$ where each $M_i$ is a module over $A \in \Calg(\mathscr{C})$. Moreover, $\Mod(\mathscr{C})^{\otimes}$ comes naturally equipped with the obvious coCartesian fibration \begin{align*}
p_{\Mo}:  \Mod(\mathscr{C})^{\otimes} & \rightarrow \Calg(\mathscr{C}) \times \mathrm{N}(\mathscr{F}\mathrm{in}_*)\\
[A,M_1,\ldots, M_n] &\mapsto (A,\left< n \right>).\end{align*}

\section{Representations of differential graded Lie algebras} The materials here are essentially taken from \cite{6}.

\begin{defi}\label{d2.1} Let $\mathfrak{g}_*$ be a dgLa over a field $k$. A \emph{representation} of $\mathfrak{g}_*$ is a differential graded vector space $V_*$, equipped with a map of chain complexes of vector spaces
$$\mathfrak{g}_*\otimes_k V_* \rightarrow V_*$$ such that 
$$[x,y]v=x(yv)+(-1)^{pq} y(xv)$$ for $x\in \mathfrak{g}_p$ and $y\in\mathfrak{g}_q$. 

A morphism between two representations $V_*$ and $W_*$ of $\mathfrak{g}_*$ is a morphism of differential graded vector spaces $f: V_* \rightarrow W_*$ such that the following diagram is commutative \begin{center}
\begin{tikzpicture}[every node/.style={midway}]
  \matrix[column sep={10em,between origins}, row sep={3em}] at (0,0) {
    \node(Y){$\mathfrak{g}_*\otimes_k V_*$} ; & \node(X) {$V_*$}; \\
    \node(M) {$\mathfrak{g}_*\otimes_k W_*$}; & \node (N) {$W_*$};\\
  };
  
  \draw[->] (Y) -- (M) node[anchor=east]  {$\mathrm{Id}_{\mathfrak{g}_*}\otimes f$}  ;
  \draw[->] (Y) -- (X) node[anchor=south]  {};
  \draw[->] (X) -- (N) node[anchor=west] {$f$};
  \draw[->] (M) -- (N) node[anchor=north] {};.
\end{tikzpicture}
\end{center}

The representations of $\mathfrak{g}_*$ comprise a category which we will denote by $\Rep_{\mathfrak{g}_*}^{dg}$.
\end{defi}
\begin{rem} \label{l2.1} Sometimes, we also use the terminology ``modules over $\mathfrak{g}_*$'' or simply ``$\mathfrak{g}_*$-modules'' to refer to representations of $\mathfrak{g}_*$ if no potential confusion arises.
\end{rem}
\begin{prop} \cite[Chapter 2, Proposition 2.4.5.]{6}
The category $\Rep_{\mathfrak{g}_*}^{dg}$ of representations of a dgLa $\mathfrak{g}_*$ admits a combinatorial model structure, where:
\begin{itemize}
\item[(1)] A map $f: V_*\rightarrow W_*$ of representations of $\mathfrak{g}_*$ is a weak equivalence if and only if it is an isomorphism on cohomology.
\item[(2) ]A map $f: V_*\rightarrow W_*$ of representations of $\mathfrak{g}_*$ is a fibration if and only if it is degreewise surjective.
\end{itemize}
We denote $\rep_{\mathfrak{g}_*}$ to be the corresponding $\infty$-category of $\Rep_{\mathfrak{g}_*}^{dg}$ with respect to this model structure.
\end{prop}
\begin{defi}\cite[Construction 2.2.13]{6} The \emph{cohomological Chevalley-Eilenberg complex} of $\mathfrak{g}_*$ is defined to be the linear dual of the tensor product
$$U(\Cn(\mathfrak{g})_*)\otimes_{U(\mathfrak{g}_*)}^{\mathbb{L}}k,$$
which we shall denote by $C^*(\mathfrak{g}_*)$. Here $\mathrm{Cn}(-)$ is the mapping cone associated to the identity map and $U(-)$ is the universal enveloping algebra.
\end{defi}
There is a natural multiplication on $C^*(\mathfrak{g}_*)$. More precisely, for $\lambda \in C^p(\mathfrak{g}_*)$ and $\mu\in C^q(\mathfrak{g}_*)$, we define $\lambda\mu \in C^{p+q}(\mathfrak{g}_*)$ by the formula
$$(\lambda\mu)(x_1\cdots x_n)=\sum_{S,S'}\epsilon(S,S')\lambda(x_{i_1}\cdots x_{i_m})\mu(x_{j_1}\cdots x_{i_{n-m}}),$$ where $x_i\in \mathfrak{g}_{r_i}$, the sum is taken over all disjoint sets $S=\{ i_1< \cdots < i_m\}$ and $S'=\{j_1<\cdots <j_{n-m} \}$ and $r_{i_1}+\cdots +r_{i_m}=p$, and $\epsilon(S,S')=\prod_{i\in S',j\in S,i<j}(-1)^{r_ir_j}$. This multiplication imposes a structure of cdga on $C^*({\mathfrak{g}_*})$.

\begin{prop} \label{p2.2} \cite{Qui69}, \cite[Prop. 2.2.6, Prop. 2.2.7, Prop. 2.2.17]{6}, \cite[Proposition 4.3.5]{Por13}
With above notations, we have the following:
\begin{itemize}
\item[(1)] The construction $\mathfrak{g}_* \mapsto C^*(\mathfrak{g}_*)$ sends quasi-isomorphisms of dgLas to quasi-isomorphisms of cdgas. In particular, we obtain a functor between $\infty$-categories $\Lie_k \rightarrow \cdg_k^{op}$, which by abuse of notations, we still denote by $C^*$.

\item[(2)] The $\infty$-functor $C^*$ preserves small co-limits. Thus, $C^*$ admits a right adjoint $$D:\cdg_k^{op} \rightarrow \Lie_k$$ to which we refer as \textit{Koszul duality}. 
\item[(3)] The unit map
 $$A_*\overset{\simeq}{\rightarrow} C^*D(A_*)$$ is an equivalence in $\dg_k$.
 \item[(4)]  The following pair \[
\begin{tikzcd}[row sep=large, column sep=large]
\tikzcdset{row sep/normal=8em}
D:\dg_k \arrow[shift right, swap]{r}{}&D(\dg_k):C^*\arrow[shift right, swap]{l}{}
\end{tikzcd}
\] is an equivalence where $D(\dg_k)$ is the essential image of $\dg_k$ under the functor $D$.
\end{itemize}
\end{prop}
The Koszul duality mentioned in the above proposition plays a crucial role in the proof of the following central result in derived deformation theory. We refer the reader to \cite{6} for a complete treatment and to \cite{3} for a glimpse. Before stating this result, for the sake of completeness, we recall the definition of formal moduli problems which actually generalizes that of functors of artinian rings in the classical sense.
\begin{defi}\label{d2.3} A functor $X: \dg_k \rightarrow \ens$ is called a formal moduli problem if the following conditions are fulfilled.
\begin{itemize}
\item[$(1)$] The space $X(k)$ is contractible.
\item[$(2)$] For every pullback diagram 
\begin{center}
\begin{tikzpicture}[every node/.style={midway}]
  \matrix[column sep={10em,between origins}, row sep={3em}] at (0,0) {
    \node(Y){$R$} ; & \node(X) {$R_0$}; \\
    \node(M) {$R_1$}; & \node (N) {$R_{01}$};\\
  };
  
  \draw[->] (Y) -- (M) node[anchor=east]  {}  ;
  \draw[->] (Y) -- (X) node[anchor=south]  {};
  \draw[->] (X) -- (N) node[anchor=west] {};
  \draw[->] (M) -- (N) node[anchor=north] {};.
\end{tikzpicture}
\end{center}
in $\dg_k$, if $\pi_0(R_0)\rightarrow \pi_0(R_{01})\leftarrow\pi_0(R_1)$ are surjective, then the diagram of spaces 

\begin{center}
\begin{tikzpicture}[every node/.style={midway}]
  \matrix[column sep={10em,between origins}, row sep={3em}] at (0,0) {
    \node(Y){$X(R)$} ; & \node(X) {$X(R_0)$}; \\
    \node(M) {$X(R_1)$}; & \node (N) {$X(R_{01})$};\\
  };
  
  \draw[->] (Y) -- (M) node[anchor=east]  {}  ;
  \draw[->] (Y) -- (X) node[anchor=south]  {};
  \draw[->] (X) -- (N) node[anchor=west] {};
  \draw[->] (M) -- (N) node[anchor=north] {};.
\end{tikzpicture}
\end{center} is also a pullback diagram.
\end{itemize}
We denote the homotopy category of formal moduli problems by $\mathcal{FMP}$.
\end{defi}
\begin{rem} \label{r2.2} If $X$ is a formal moduli problem then the restriction of its connected components $\pi_0(X)$ on the subcategory of local artinian $k$-algebras with residue field $k$ is a ``good" functor of artinian rings in the classical sense (cf. \cite{9}).

\end{rem}
\begin{thm}\label{t2.1}\cite{6,11}
The functor \begin{align*}
\Psi:\Lie_k&\rightarrow \; \mathcal{FMP}\\
\mathfrak{g}_* &\mapsto \Map_{\Lie_k}(D(-),\mathfrak{g}_*) 
\end{align*} induces an equivalence of $\infty$-categories between the category $\Lie_k$ of differential graded Lie algebras and the category $\mathcal{FMP}$ of formal moduli problems. 
\end{thm}

\begin{defi}\label{def2.4} \cite[Construction 2.4.7]{6}   
 Let $\mathfrak{g}_*$ be a dgLa and $V_* \in \Rep_{\mathfrak{g}_*}^{dg}$. The \emph{cohomological Chevalley-Eilenberg complex of} $\mathfrak{g}_*$ \emph{with coefficients in} $V_*$, denoted by $C^{*}(\mathfrak{g}_*,V_*)$, is defined to be the differential graded vector space of $U(\mathfrak{g}_*)$-module maps from $U(\Cn(\mathfrak{g})_*)$ into $V_*$.
\end{defi}

Observe that $C^{*}(\mathfrak{g}_*,V_*)$ has the structure of a module over the differential graded algebra $C^*(\mathfrak{g}_*)$. The action is given by $k$-bilinear maps
$$C^p(\mathfrak{g}_*)\times C^q(\mathfrak{g}_*,V_*)\rightarrow C^{p+q}(\mathfrak{g}_*,V_*)$$ which send $\lambda\in C^*(\mathfrak{g}_*)$ and $\mu\in C^q(\mathfrak{g}_*,V_*)$ to the element $\lambda\mu \in C^{p+q}(\mathfrak{g}_*,V_*)$ provided by

$$(\lambda\mu)(x_1\cdots x_n)=\sum_{S,S'}\epsilon(S,S')\lambda(x_{i_1}\cdots x_{i_m})\mu(x_{j_1}\cdots x_{i_{n-m}}),$$ as in the construction of the multiplication on $C^*(\mathfrak{g}_*)$.

Let $\mathfrak{g}_*$ be a dgLa and $\Mo^{dg}_{C^*(\mathfrak{g}_*)}$ be the category of differential graded modules over $C^*(\mathfrak{g}_*)$.
\begin{thm}\label{t2.2}\cite[Proposition 2.4.10 and Remark 2.4.11]{6} The functor \begin{align*}
C^*(\mathfrak{g}_*,-) : \Rep_{\mathfrak{g}_*}^{dg}&\rightarrow \Mo_{C^*(\mathfrak{g}_*)}^{dg} \\
V_* &\mapsto C^{*}(\mathfrak{g}_*,V_*)  
\end{align*}
 preserves weak equivalences and fibrations. Moreover, it has a left adjoint $F$ given by
 \begin{align*}
F :\Mo_{C^*(\mathfrak{g}_*)}^{dg} &\rightarrow \Rep_{\mathfrak{g}_*}^{dg} \\
M_* &\mapsto U(\Cn(\mathfrak{g})_*) \otimes_{C^*(\mathfrak{g}_*)}M_*.
\end{align*}  Thus, $C^*(\mathfrak{g}_*,-)$ is a right Quillen functor, which induces a map between $\infty$-categories $\rep_{\mathfrak{g}_*}$ and $\Mod_{C^*(\mathfrak{g}_*)}$.
 
\end{thm}

\begin{defi}\label{d2.5} Let $\mathfrak{g}_*$ be a dgLa and $V_*$ be a representation of $\mathfrak{g}_*$. 

\begin{enumerate}
\item[(1)] $V_*$ is said to be \emph{connective} if its image in $\Mod_k$ is connective: that is, the cohomology groups of the chain complex $V_*$ are concentrated in non-positive degrees. 

\item[(2)] $V_*$ is said to be \emph{cohomologically finite-dimensional} if its image in $\Mod_k$ has finite-dimensional cohomologies.

\item[(3)] $V_*$ is said to be \emph{bounded above} if its image in $\Mod_k$ is a bounded-above complex of $k$-modules.
\end{enumerate}
Let $\Mod^{cn}_{\mathfrak{g}_*}$ denote the full subcategory of $\rep_{\mathfrak{g}_*}$ spanned by the connective $\mathfrak{g}_*$-modules.
\end{defi}

\begin{thm}\label{t2.3}\cite[Proposition 2.4.16]{6} For $\mathfrak{g}_*=D(A_*)$ where $A_* \in \dg_k$, let $f$ be the corresponding $\infty$-functor of $F$ in Theorem \ref{t2.2}, then $f$ induces an equivalence of $\infty$-categories
$$\Mod_{A_*}^{cn}=\Mod_{C^*(\mathfrak{g}_*)}^{cn}\rightarrow \Mod_{\mathfrak{g}_*}^{cn}$$
which sends $M_*$ to $$U(\Cn(\mathfrak{g})_*) \otimes_{C^*(\mathfrak{g}_*)}^{\mathbb{L}} M_*,$$
i.e. $f$ is the left derived functor of $F$.
\end{thm}

To end this section, we recall a little bit about tensor products of representations. 
\begin{defi}\label{d1.9} Let $V_*$ and $W_*$ be two representations of $\mathfrak{g}_*$, then tensor product $V_*\otimes_k W_*$ can be considered a representation of $\mathfrak{g}_*$ with action given by the formula
$$x(v\otimes w)=(xv)\otimes w+(-1)^{pq}v\otimes (xw)$$ for homogeneous elements $x\in \mathfrak{g}_p, v\in V_q$ and $w\in W_r$. 
\end{defi}
By a general theorem of Lurie, we can prove that the construction $$W_*\mapsto V_*\otimes_kW_*$$ preserves quasi-isomorphisms. Consequently, the $\infty$-category $\rep_{\mathfrak{g}_*}$ inherits a symmetric monoidal structure. 
\begin{const} We create a new category denoted by ${\Rep_{dg}^{\otimes}}^{\otimes}$ as follows.
\begin{enumerate}
\item[(1)]  An object of ${\Rep_{dg}^{\otimes}}^{\otimes}$ is a tuple $(\mathfrak{g}_*, V^1_*,\ldots,V^n_*)$ where $\mathfrak{g}_*$ is a dgLa and each $V_*^i \in \Rep_{\mathfrak{g}_*}$.

\item[(2)] A morphism between two objects $(\mathfrak{g}_*, V^1_*,\ldots,V^m_*)$ and $(\mathfrak{h}_*, W^1_*,\ldots,W^n_*)$   in ${\Rep_{dg}^{\otimes}}^{\otimes}$ is the datum of a map $\alpha: \left<m \right> \rightarrow \left< n\right>$ in $\mathscr{F}\mathrm{in}_*$, a morphism $\phi: \mathfrak{h}_* \rightarrow \mathfrak{g}_*$ of differential graded Lie algebras  and a collection of morphisms $\lbrace f_j: \bigotimes _{\alpha(i)=j} V_*^i \rightarrow W_*^j\rbrace_{1\leq j\leq n}$ of representations of $\mathfrak{h}_*$ where each $V_*^i$ is thought of as a representation of $\mathfrak{h}_*$ by means of the morphism $\phi$.
\end{enumerate}

\end{const}
The category ${\Rep_{dg}^{\otimes}}^{\otimes}$ comes equipped naturally with a forgetful functor  \begin{align*}
 {\Rep_{dg}^{\otimes}}^{\otimes} & \rightarrow \Li_k^{op}\times \mathscr{F}\mathrm{in}_* \\
[\mathfrak{g}_*, V_*^1,\ldots, V_*^n] &\mapsto ( \mathfrak{g}_*, \left< n \right>),\end{align*} 
which induces a coCartesion fibration $\mathrm{N}({\Rep_{dg}^{\otimes}}^{\otimes}) \rightarrow \mathrm{N}(\Li_k)^{op}\times \mathrm{N}(\mathscr{F}\mathrm{in}_*)$. Thus, we have the following.

\begin{prop} \label{p2.3} \cite[Corollary 2.4.20]{6} Let $W$ be the collection of all morphisms in ${\Rep_{dg}^{\otimes}}^{\otimes}$ of the form 
$$\alpha:(\mathfrak{g}_*, V^1_*,\ldots,V^n_*)\rightarrow (\mathfrak{g}_*, W^1_*,\ldots,W^n_*) $$ where the image of $\alpha$ in $\Li_k$ and $\mathscr{F}\mathrm{in}_*$ is an identity and $\alpha$ induces a quasi-isomorphism $V^i_* \rightarrow W_*^i$ for $1\leq i \leq n$. Then, we have a coCartesian fibration ${\Rep_{dg}^{\otimes}}^{\otimes}[W^{-1}] \rightarrow \mathrm{N}(\Li_k)^{op}\times \mathrm{N}(\mathscr{F}\mathrm{in}_*)$ of which the fiber over any given differential graded Lie algebra $\mathfrak{g}_*$ can be identified with the symmetric monoidal $\infty$-category $\rep_{\mathfrak{g}_*}$.
 \end{prop}
\section{Quasi-coherent sheaves on formal moduli problems} In this section, the notion of quasi-coherent sheaves on formal moduli problems is briefly introduced. The interested reader is referred to \cite[Chapter 2.4]{6} for more details. 

Recall first that $\cat$ is the $\infty$-category of small $\infty$-categories and $\widehat{\cat}$ is that of $\infty$-categories not necessarily small (cf. \cite[Chapter III]{8}). Let $\Mod_k$ be the $\infty$-category of chain complexes of $k$-modules. By the construction at the end of $\S2$, we obtain a coCartesian fibration $\Mod(\Mod_k)^{\otimes}  \rightarrow \cdg_k \times \mathrm{N}(\mathscr{F}\mathrm{in}_*)$ being further classified by a map $\chi: \cdg_k \rightarrow \Calg(\widehat{\cat})$ which to each commutative differential graded algebra $A_*$, associates the symmetric monoidal $\infty$-category $\Mod_{A_*}$ of modules over $A_*$. More precisely, the map $\chi$ is just the image of the coCartesian fibration under the straightening functor given in \cite[Theorem 3.2.0.1]{8} (see also \cite{HHR25} for a more convenient version). Restricting this functor on the full sub-category $\dg_k$ of $\cdg_k$, we obtain another one denoted by $\chi^{\mathrm{sm}}$. An easy application of \cite[Theorem 5.1.5.6]{8} gives us an essentially unique fatorization of $\chi^{\mathrm{sm}}$ as a composition

\begin{center}
\begin{tikzpicture}[every node/.style={midway}]
  \matrix[column sep={12em,between origins}, row sep={4em}] at (0,0) {
    \node(Y){$\cdg_k$} ; & \node(X) {$\Calg(\widehat{\cat})$}; \\
    \node(M) {$\dg_k$}; & \node (N) {$\Fun(\dg_k,\ens) ^{op}$};\\
  };
  
  \draw[->] (M) -- (Y) node[anchor=east]  {}  ;
  \draw[->] (Y) -- (X) node[anchor=south]  {$\chi$};
  \draw[->] (N) -- (X) node[anchor=west] {$\qc$};
  \draw[->] (M) -- (N) node[anchor=north] {$\mathcal{Y}$};.
  \draw[->] (M) -- (X) node[anchor=south] {$\chi^{\mathrm{sm}}$};.
\end{tikzpicture}
\end{center}
 where  $\mathcal{Y}$ is the Yoneda embedding functor. We denote the rightmost vertical one by $\qc$. Note that $\qc$ also preserves colimits and for every functor $X: \dg_k \rightarrow \ens$ (in particular, a formal moduli problem), we obtain a symmetric monoidal $\infty$-category $\qc(X)$ to which we refer the $\infty$-category of quasi-coherent sheaves on $X$. We recollect some fundamental properties of $\qc$.

\begin{prop} \label{p3.1}\cite[Remark 2.4.26]{6} \begin{enumerate}
\item[(a)] For every functor $X: \dg_k \rightarrow \ens$, the $\infty$-category $\qc(X)$ is presentable and the tensor product $\otimes: \qc(X) \times \qc(X) \rightarrow \qc(X)$ preserves small colimits separately in each variable.
\item[(b
)] Every natural transformation $f: X \rightarrow Y$ of functors in $\Fun(\dg_k,\ens)$ induces a functor $f^*: \qc(Y) \rightarrow \qc(X)$ preserving small colimits.
\end{enumerate}
\end{prop}

\begin{rem}\label{r3.1} Intuitively, if $X$ is a formal moduli problem, then a quasi-coherent sheaf $\mathscr{F}\in\qc(X)$ is a rule which to each point $\mu \in X(A_*)$ associates an $A_*$-module $\mathscr{F}_{\nu}$ and to each morphism $f: A_* \rightarrow A'_*$ which sends $\mu$ to $\mu' \in X(A'_*)$ associates an equivalence $\mathscr{F}_{\mu'} \cong 
A'_*\otimes_{A_*} \mathscr{F}_\mu$. In particular, when $X$ is representable by an object $A_* \in \dg_k$, then $\qc(X)$ is nothing but the $\infty$-category $\Mod_{A_*}$ of modules over $A_*$. 
\end{rem}
Once again by \cite{8,HHR25}, the coCartesian fibration  ${\Rep_{dg}^{\otimes}}^{\otimes}[W^{-1}] \rightarrow \mathrm{N}(\Li_k)^{op}\times \mathrm{N}(\mathscr{F}\mathrm{in}_*)$ in Proposition \ref{p2.3} is classified also by a functor 
\begin{align*}
\bar{\chi} :\Lie_k^{op} &\rightarrow \Calg(\widehat{\cat})\\
\mathfrak{g}_* &\mapsto \rep_{\mathfrak{g}_*}.
\end{align*}
Composing $\bar{\chi}$ with the Koszul functor $D$ in Proposition \ref{p2.2}, we obtain another functor from $ \dg_k$ to $ \Calg(\widehat{\cat})$, denoted by $\chi_{!}^{\mathrm{sm}}$. Once again, applying \cite[Theorem 5.1.5.6]{8} provides an essentially unique factorization of $ \chi_{!}^{\mathrm{sm}}$
\begin{center}
\begin{tikzpicture}[every node/.style={midway}]
  \matrix[column sep={12em,between origins}, row sep={4em}] at (0,0) {
    \node(Y){$\Lie_k^{op}$} ; & \node(X) {$\Calg(\widehat{\cat})$}; \\
    \node(M) {$\dg_k$}; & \node (N) {$\Fun(\dg_k,\ens) ^{op}$};\\
  };
  
  \draw[->] (M) -- (Y) node[anchor=east]  {$D$}  ;
  \draw[->] (Y) -- (X) node[anchor=south]  {$\bar{\chi}$};
  \draw[->] (N) -- (X) node[anchor=west] {$\qc_!$};
  \draw[->] (M) -- (N) node[anchor=north] {$\mathcal{Y}$};.
  \draw[->] (M) -- (X) node[anchor=south] {$\chi^{\mathrm{sm}}_{!}$};.
\end{tikzpicture}
\end{center}
Like $\qc$, the functor $\qc_!$ enjoys  the properties in Proposition \ref{p3.1} as well.

A remark is in order. For any $A_*\in\dg_k$, by Proposition \ref{p2.2}, the unit map $A_* \rightarrow C^*(D(A_*))$ is an equivalence. Therefore, the functor $\chi^{\mathrm{sm}}$ fits also into the following commutative diagram
\begin{center}
\begin{tikzpicture}[every node/.style={midway}]
  \matrix[column sep={12em,between origins}, row sep={4em}] at (0,0) {
\node(A){$\Lie_k^{op}$} ; &    \node(Y){$\cdg_k$} ; & \node(X) {$\Calg(\widehat{\cat})$}; \\
 \node(B){} ; &   \node(M) {$\dg_k$}; & \node (N) {$\Fun(\dg_k,\ens) ^{op}$.};\\
  };
  \draw[->] (A) -- (Y) node[anchor=south]  {$C^*$}  ;
  \draw[->] (A) to[bend left] node[midway,above] {$\bar{\chi}$} (X) ;
  \draw[->] (M) -- (Y) node[midway]  {}  ;
  \draw[->] (M) -- (A) node[anchor=south]  {$D$}  ;
  \draw[->] (Y) -- (X) node[anchor=south]  {$\chi$};
  \draw[->] (N) -- (X) node[anchor=west] {$\qc$};
  \draw[->] (M) -- (N) node[anchor=north] {$\mathcal{Y}$};.
  \draw[->] (M) -- (X) node[anchor=south] {$\chi^{\mathrm{sm}}$};.
\end{tikzpicture}
\end{center} The functor $f$ in Theorem \ref{t2.3} defines a natural transformation of functors $\chi^{\mathrm{sm}} \rightarrow \chi^{\mathrm{sm}}_!$ which for each $A_* \in \dg_k$, induces a fully faithful embedding whose restriction on connective objects is an $\infty$-equivalence.
\begin{center}
\begin{tikzpicture}[every node/.style={midway}]
  \matrix[column sep={12em,between origins}, row sep={4em}] at (0,0) {
    \node(Y){$\chi^{\mathrm{sm}}(A_*)\simeq \Mod_{A_*}$} ; & \node(X) {$\chi^{\mathrm{sm}}_!(A_*)\simeq\rep_{D(A_*)}$}; \\
    \node(M) {$\Mod_{A_*}^{cn}$}; & \node (N) {$\Mod_{D(A_*)}^{cn}.$};\\
  };
  
  \draw[->] (M) -- (Y) node[anchor=east]  {}  ;
  \draw[right hook->] (Y) -- (X) node[anchor=south]  {};
  \draw[->] (N) -- (X) node[anchor=west] {};
  \draw[->] (M) -- (N) node[anchor=south] {$\simeq$};.
 
\end{tikzpicture}
\end{center}
Consequently, we attain also a natural transformation between $\qc$ and $\qc_!$ and for each $X \in \Fun(\dg_k,\ens) ^{op} $, the following fully faithful embedding is available
$$\qc(X) \hookrightarrow \qc_!(X).$$ Surprisingly, when $X$ is a formal moduli problem, $\qc_!(X)$ has an extremely explicit presentation: they are merely differential graded modules over its associated differential graded Lie algebra (cf. \cite[Proposition 2.4.31]{6}).
\begin{prop} \cite[Proposition 2.4.31]{6} Let $X$ be a formal moduli problem and $\mathfrak{g}_*$ its associated governing differential graded Lie algebra. Then there exists a canonical equivalence $$\qc_!(X)\simeq \rep_{\mathfrak{g}_*}$$ 
 of symmetric monoidal $\infty$-categories
\end{prop}

Finally, we introduce the notion of connectivity on quasi-coherent sheaves. Let $X$ be a formal moduli problem and $\mathscr{F}$ be in $\qc(X)$. $\mathscr{F}$ is said to be \emph{connective} if for any $A\in \dg_k$ and any $\mu \in X(A)$ each ``stalk'' $\mathscr{F}_{\mu} \in \Mod_{A}$ is connective. If we denote by $\qc(X)^{cn}$ the full subcategory  of $\qc(X)$ generated by the connective objects, $\qc(X)^{cn}$ is a presentable $\infty$-category closed under colimits and extension in $\qc(X)$. This section ends with the following result.
\begin{prop}\cite[Proposition 2.4.34]{6}\label{p3.3} Let $X$ be a formal moduli problem and $\mathfrak{g}_*$ its associated governing differential graded Lie algebra. Then there exists a canonical equivalence  $$\qc(X)^{cn} \overset{\simeq}{\rightarrow} \Mod_{\mathfrak{g}_*}^{cn}$$ of symmetric monoidal $\infty$-categories.
\end{prop}
\begin{rem} Through the rest of the paper, by quasi-coherent sheaves, we really mean connective quasi-coherent sheaves.
\end{rem}
\section{Cohomology jump loci and differential graded Lie algebras} \label{s4} The main reference for this section is \cite{1} to which the interested reader is referred.
\subsection{Cohomology jump loci of complexes} Let $R$ be a noetherian ring and let $E^{\bullet}$ be an above-bounded complex of $R$-modules, whose cohomologies are finitely generated $R$-modules. It is well-known that there exists a bounded above complex of finitely generated free $R$-modules which is
 quasi-isomorphic to $E^{\bullet}$. Let $F^{\bullet}$ be such a complex. For each $i$, consider the map $$d_{i-1}\oplus d_{i}:  F^{i-1}\oplus F^{i} \rightarrow F^i\oplus F^{i+1}.$$

\begin{defi}\label{d4.1} \cite[Definition-Proposition 2.2]{1} The cohomology jump ideals are defined to be $$J_{k}^i(E^{\bullet})=I_{\mathrm{rank}(F^i)-k+1}(d_{i-1}\oplus d_{i})$$ where $I_{\mathrm{rank}(F^i)-k+1}(d_{i-1}\oplus d_{i})$ is the ideal of $R$, generated by minors of size $\mathrm{rank}(F^i)-k+1$ of the representing matrix of $d_{i-1}\oplus d_{i}$. Moreover, $J_{k}^i(E^{\bullet})$ is independent of the choice of $F^{\bullet}$ and then gives a homotopic invariant of $E^{\bullet}$.

\end{defi}
For complexes of flat $R$-modules, cohomology jump ideals are functorial with respect to base change. This leads to several interesting properties listed in the following proposition.

\begin{prop} \label{p4.1}\cite[Corollary 2.4, Corollary 2.5]{1} Let $E^{\bullet}$ be a chain complex of flat $R$-modules. Then
\begin{enumerate}
\item[(1)]For any noetherian $R$-algebra $S$, $J_{k}^i(E^{\bullet})\cdot S = J_{k}^i(E^{\bullet}\otimes_R S)$.

\item[(2)] In the case that $R$ is a field, $$J_{k}^i(E^{\bullet}) =\begin{cases}
 0 & \text{ if } \dim H^i(E^{\bullet}) \geq k \\ R & \text{ if } \dim H^i(E^{\bullet}) < k.
\end{cases}$$

\item[(3)] For any maximal ideal $\mathfrak{m}$ of $R$, $J_{k}^i(E^{\bullet}) \subset \mathfrak{m}$ if and only if $\dim_{R/ \mathfrak{m}} H^i(E^{\bullet}\otimes_R R/ \mathfrak{m})\geq k$.
\end{enumerate} 

\end{prop}
\subsection{Cohomology jump functors } Let $\mathfrak{g}_*$ be a differential graded Lie algebra,  we recall the classical deformation functor $\mathbf{\mathrm{MC}}_{\mathfrak{g}_*}$ associated to $\mathfrak{g}_*$, defined via the Maurer-Cartan equation. We have two functors:
\begin{itemize}
\item[(1)] The Gauge functor 
\begin{align*}
G_{\mathfrak{g}_*}: \Art_\mathbb{C} & \rightarrow \mathrm{Grp} \\
A &\mapsto \mathrm{exp}(\mathfrak{g}_0\otimes \mathbf{\mathrm{m}}_A)
\end{align*} where $\mathbf{\mathrm{m}}_A$ is the unique maximal ideal of $A$ and $\mathrm{Grp}$ is the category of groups.

\item[(2)] The Maurer-Cartan functor $MC_{\mathfrak{g}_*}: \Art_k \rightarrow \se $ defined by
\begin{align*}
MC_{\mathfrak{g}_*}: \Art_k & \rightarrow \se \\
A & \mapsto \left \{ x \in \mathfrak{g}_1\otimes \mathrm{m}_A\mid dx+\frac{1}{2}[x,x]=0  \right \}.
\end{align*} 
\end{itemize}
For each $A$, the gauge action of $G_{\mathfrak{g}_*}(A)$ on the set $MC_{\mathfrak{g}_*}(A)$ is functorial in $A$ and gives an action of the group functor  $G_{\mathfrak{g}_*}$ on $MC_{\mathfrak{g}_*}$. This allows us to define the quotient functor \begin{align*}
\mathbf{\mathrm{MC}}_{\mathfrak{g}_*}: \Art_k & \rightarrow \se \\
A & \mapsto MC_{\mathfrak{g}_*}(A)/G_{\mathfrak{g}_*}(A),
\end{align*} 
\begin{defi}\label{d4.2}\cite[Definition 3.8]{1} Let $V_*$ be a connective $\mathfrak{g}_*$-module. For any $A \in \Art_k$ and any $\omega\in \mathbf{\mathrm{MC}}_{\mathfrak{g}_*}(A)$, the Aomoto complex is defined to be the chain complex of $A$-modules $V_*^{A,\omega}:=A\otimes _{k}V_*$ with the differential given by $d_{\omega}:=\mathrm{id}_A \otimes  d +\omega$.
\end{defi} 
\begin{rem}\label{r4.1} In fact, the Aomoto complexes can even be defined for any $A_*\in \dg_k$ in the same manner as in Definition \ref{d4.2} when we extend the deformation functor $\mathbf{\mathrm{MC}}_{\mathfrak{g}_*}$ to a functor defined over $\dg_k$.
\end{rem}

\begin{defi}\label{d4.3} \cite[Definition 3.15]{1} The \emph{cohomology jump functor} associated to $\mathbf{\mathrm{MC}}_{\mathfrak{g}_*}$ and a connective $\mathfrak{g}_*$-module $V_*$ is defined to be 
\begin{align*}
\mathbf{\mathrm{MC}}_{\mathfrak{g}_*,k}^{V_*,i}: \Art_k & \rightarrow \se \\
A & \mapsto \lbrace \omega \in \mathbf{\mathrm{MC}}_{\mathfrak{g}_*}(A)\mid J_{k}^i(A \otimes_k V_*,d_{\omega})=0  \rbrace,
\end{align*} 
 
\end{defi}
Once again, the curious reader is referred to \cite[Section 3]{1} for the well-definedness of these functors and related properties.
\subsection{Derived cohomology jump functors and fmps with cohomological constraints} We propose a derived version of cohomology jump functors attached to a given pair $(V_*, \mathfrak{g}_*)$ defined previously where $\mathfrak{g}_*$ is a differential graded Lie algebra of which $V_*$ is some representation. It turns out that the restriction of the simplicial ones on the category $\Art_k$ of local artinian $k$-algebras with residue field $k$ gives back the classical ones.
\begin{rem} Each time a functor $F:\mathscr{C}\rightarrow \mathscr{D}$  descends to a functor on the associated $\infty$-category of $\mathscr{C}$, by abuse of notations, we still use $F$ to denote its left derived functor.
\end{rem}

On one hand, the classical deformation functor $\mathbf{\mathrm{MC}}_{\mathfrak{g}_*}$ can be naturally extended to a formal moduli problem in Lurie's sense (cf. Definition \ref{d2.3}) via a simplicial version of the Maurer-Cartan equation (see \cite{4} for such a construction). In other words, we have a fmp \begin{equation}\label{e3.1} \mathfrak{MC}_{\mathfrak{g}_*}: \dg_k \rightarrow \ens 
\end{equation} such that \begin{equation} \label{coco}\pi_0(\mathfrak{MC}_{\mathfrak{g}_*})=\mathbf{\mathrm{MC}}_{\mathfrak{g}_*}.
\end{equation} On the other hand, there is an equivalence \begin{equation}\label{e3.2} \mathfrak{MC}_{\mathfrak{g}_*} \rightarrow \Map_{\Lie_{k}}(D(-),\mathfrak{g}_*) 
\end{equation}as formal moduli problems (cf. \cite[$\S2$]{6} and \cite{4}). Consequently, $\Map_{\Lie_{k}}(D(-),\mathfrak{g}_*)$ can be thought of as a natural extension of $\mathbf{\mathrm{MC}}_{\mathfrak{g}_*}$ in the derived world.

\begin{const} Let $X$ be a formal moduli problem and $\mathfrak{g}_*$ be its associated governing differential graded Lie algebra. Given a quasi-coherent sheaf $\mathscr{F} \in \qc(X)^{cn}$ (see Remark \ref{r3.1}), the ``stalk'' at the base point $k$ provides a chain complex of $k$-modules whose $\mathfrak{g}_*$-module structure is given by means of the isomorphism in Proposition \ref{p3.3}. Conversely, let $V_*$ be a $\mathfrak{g}_*$-module. The associated quasi-coherent sheaf $\mathscr{F}^{V_*}$ is constructed as follows. For the sake of the equivalence $$X \simeq \Map_{\Lie_k}(D(-),\mathfrak{g}_*)$$ where $D$ is the Koszul duality introduced in Proposition \ref{p2.2} and Theorem \ref{t2.1}, for any $A_* \in \dg_k$, each $\eta \in X(A_*)$ corresponds to a morphism of differential graded Lie algebras $\bar{\eta}: D(A_*)\rightarrow \mathfrak{g}_*$. Composing this with the structural morphism $\mathfrak{g}_* \rightarrow \mathrm{End}(V_*)$, we obtain a $D(A_*)$-module $V_*$. Finally, an easy application of Theorem \ref{t2.3} gives us an $A_*$-module denoted by $\mathscr{F}^{V_*}_{\eta}$. Suppose further that we have a morphism $f: A_* \rightarrow A'_*$ sending $\eta$ to $\eta' \in X(A'_*)$ then the following diagram is available
\begin{center}
\begin{tikzpicture}[every node/.style={midway}]
  \matrix[column sep={10em,between origins}, row sep={4em}] at (0,0) {
    \node(Y){$D(A_*)$} ; & \node(X) {$\mathfrak{g}_*$}; & \node(K) {$V_*$}; \\
    \node(M) {$D(A'_*)$}; & \node (N) {};\\
  };
  
  \draw[->] (M) -- (Y) node[anchor=east]  {$D(f)$}  ;
    \draw[dashed,->] (Y) to[bend left=20] (K) node[anchor=south] {};
  \draw[->] (X) -- (K) node[anchor=east]  {}  ;
  \draw[->] (Y) -- (X) node[anchor=south]  {$\bar{\eta}$};
  \draw[dashed,->] (M) to[bend right=15] (K) node[anchor=south] {};
  \draw[->] (M) -- (X) node[anchor=south] {$\bar{\eta'}$};.
\end{tikzpicture}
\end{center}
from which we deduce that the $D(A'_*)$-module structure of $V_*$ is induced from the $D(A_*)$-one via the morphism $D(f)$. Therefore, the  equivalence $$\mathscr{F}^{V_*}_{\eta'}\simeq A'_*\otimes_{A_*} \mathscr{F}^{V_*}_{\eta} $$ of $A'_*$-modules follows from the functoriality of $C^*(\mathfrak{g}_*,-)$ in Theorem \ref{t2.2} and Theorem \ref{t2.3}. This proves that the object $\mathscr{F}^{V_*}$ built up in this way is indeed a quasi-coherent sheaf on $X$ (or even a connective quasi-coherent sheaf on $X$ if the initial $\mathfrak{g}_*$-module is connective). We refer to this whole construction of $\mathscr{F}^{V_*}$ the \emph{sheafification} of the $\mathfrak{g}_*$-module $V_*$.
\end{const}
Moreover, we shall prove that this construction is closely related to  the one of Aomoto complexes given in Definition \ref{d4.2} and Remark \ref{r4.1}  by  using the following chain of isomorphisms
\begin{equation}\label{e4.3} X \simeq \Map_{\Lie_k}(D(-),\mathfrak{g}_*) \simeq \mathfrak{MC}_{\mathfrak{g}_*}
\end{equation} of formal moduli problems (cf. \cite[$\S2$]{6} and \cite{4}). This is the content of the proposition below. 
\begin{prop} \label{p4.2} Let $A_* \in \dg_k$, $\omega \in  X(A_*)$ and $V_*$ a connective $\mathfrak{g}_*$-module. Then there exists an equivalence 
$$\mathscr{F}^{V_*}_{\omega} \overset{\simeq}{\rightarrow} V_*^{A_*,\omega}:=(A_*\otimes _{k}V_*,d_{\omega})$$ of $A_*$-modules.

\end{prop}
\begin{proof} The image of $(A_*\otimes _{k}V_*,d_{\omega})$ under the functor $f$ in Theorem \ref{t2.3} is a connective $D(A_*)$-module in $\Mod_{D(A_*)}^{cn}$.  Now, unwinding the definition of $f$, we have $$f(A_* \otimes_k V_*,d_{\omega})=U(\Cn(D(A_*))) \otimes_{A_*}^\mathbb{L} (A_* \otimes_k V_*,d_{\omega}).$$
Here, we used the fact that for any $A_*\in \dg_k$, the unit morphism $$C^*D(A_*) \rightarrow A_*$$ is an equivalence in $\dg_k$ (cf. Proposition \ref{p2.2}). The cone $\Cn(D(A_*))$ of $D(A_*)$ is a contractible chain complex since its underlying chain complex can be identified with the mapping cone of the identity $D(A_*)\rightarrow D(A_*)$. In particular, $0\rightarrow \Cn(D(A_*)) $ is a quasi-isomorphism of dgLas. Because the universal enveloping algebra construction preserves quasi-isomorphisms, $U(0)=k\rightarrow U(\Cn(D(A_*))$ is also a weak equivalence. Thus, \begin{align*}
f(A_* \otimes_k V_*,d_{\omega}) &\simeq U(\Cn(D(A_*))) \otimes_{A_*}^{\mathbb{L}} (A_* \otimes_k V_*,d_{\omega}) \\
 &\simeq k \otimes_{A_*}^{\mathbb{L}} (A_* \otimes_k V_*,d_{\omega})\\
 &\simeq V_* 
\end{align*}
On one hand, by the isomorphisms (\ref{e4.3}) the element $\omega \in \mathfrak{MC}_{\mathfrak{g}_*}(A_*) $ corresponds to a morphism $\bar{\omega}: D(A_*)\rightarrow \mathfrak{g}_*$ which defines the $D(A_*)$-module structure of $f(V_*\otimes _{k}A_*,d_{\omega})$. On the other hand, by abuse of notations, we can suppose that $\omega \in X(A_*)$ which corresponds to the same morphism $\bar{\omega}: D(A_*)\rightarrow \mathfrak{g}_*$ up to equivalence once again by the isomorphisms (\ref{e4.3}). Therefore, the $D(A_*)$-module structure of $f(A_* \otimes_k V_*,d_{\omega})$ and that of $f(\mathscr{F}^{V_*}_{\omega})$ are the same up to equivalence. Moreover, they have the same underlying chain complex up to quasi-isomorphism of chain complexes. Consequently, they are equivalent in $\Mod_{D(A_*)}^{cn}$.
\end{proof}

Next, we introduce the notion of \emph{coherent sheaves} and that of \emph{flat sheaves} on formal moduli problems. Let $A_*$ be in $\dg_k$, then $A_*$ is cohomologically concentrated in non-positive degrees. By a simple canonical truncation argument, we can even suppose that $A_*$ is really concentrated in non-positive degrees i.e.
$$A_*= (\cdots \overset{d}{\rightarrow}A_{-n}\overset{d}{\rightarrow} A_{-n+1}\overset{d}{\rightarrow} \cdots \overset{d}{\rightarrow}A_{-1}\overset{d}{\rightarrow} A_0\overset{d}{\rightarrow}0\overset{d}{\rightarrow} \cdots ).$$ Moreover, the space $d(A_{-1})$ of boundaries generates an ideal in $A_0$ so that the canonical projection $$A_0 \rightarrow H^0(A_*)=A_0/d(A_{-1})$$ becomes a morphism of $k$-algebras. If we think of $H^{0}(A_*)$ as a differential graded local artinian $k$-algebra concentrated in degree $0$ then obtain a morphism \begin{center}
\begin{tikzpicture}[every node/.style={midway}]
  \matrix[column sep={5.5em,between origins}, row sep={4em}] at (0,0) {
    \node(A1){$A_*$} ; & \node(A2) {$\cdots$}; & \node(A3){$A_{-n}$} ; & \node(A4) {$A_{-n+1}$}; & \node(A5){$\cdots$} ; & \node(A6) {$A_{-1}$};& \node(A7){$A_0$} ; & \node(A8) {$\cdots$}; \\
    \node(B1){$H^0(A_*)$} ; & \node(B2) {$\cdots$}; & \node(B3){$0$} ; & \node(B4) {$0$}; & \node(B5){$\cdots$} ; & \node(B6) {$0$};& \node(B7){$H^0(A_*)$} ; & \node(B8) {$\cdots$};\\
  };
  
  \draw[->] (A2) -- (A3) node[anchor=south]  {$d$}  ;
  \draw[->] (A3) -- (A4) node[anchor=south]  {$d$};
  \draw[->] (A4) -- (A5) node[anchor=south] {};
  \draw[->] (A5) -- (A6) node[anchor=south] {$d$};
   \draw[->] (A6) -- (A7) node[anchor=south] {$d$}; \draw[->] (A7) -- (A8) node[anchor=south] {$d$};
   
   \draw[->] (B2) -- (B3) node[anchor=south]  {};
  \draw[->] (B3) -- (B4) node[anchor=south]  {};
  \draw[->] (B4) -- (B5) node[anchor=south] {};
  \draw[->] (B5) -- (B6) node[anchor=south] {};
   \draw[->] (B6) -- (B7) node[anchor=south] {}; 
   \draw[->] (B7) -- (B8) node[anchor=south] {};
   \draw[->>] (A1) -- (B1) node[anchor=south]  {};
   \draw[->>] (A3) -- (B3) node[anchor=south] {};
   \draw[->>] (A4) -- (B4) node[anchor=south]  {};
   \draw[->>] (A6) -- (B6) node[anchor=south]
{};
   \draw[->>] (A7) -- (B7) node[anchor=south] {};.
\end{tikzpicture}
\end{center} in $\dg_k$. This gives rise to a functor $\mathscr{H}_{A_*}: \Mo_{A_*} \rightarrow \Mo_{H^0(A_*)}$ which to each $A_*$-module $V_*$, associates the $H^{0}(A_*)$-module $H^0(A_*)\otimes_{A_*} V_* $. Moreover, $\mathscr{H}_{A_*}$ can be shown to send equivalences of $A_*$-modules to those of $H^0(A_*)$-modules. Therefore, it descends to a functor $\Mod_{A_*} \rightarrow \Mod_{H^0(A_*)}$ of $\infty$-categories, which by abuse of notation we still denote by $\mathscr{H}_{A_*}$. Furthermore, for any morphism $A \rightarrow A'$ in $\dg_k$, the following commutative diagram 
\begin{center}
\begin{tikzpicture}[every node/.style={midway}]
  \matrix[column sep={12em,between origins}, row sep={4em}] at (0,0) {
    \node(Y){$A_*$} ; & \node(X) {$A_*'$}; \\
    \node(M) {$H^0(A_*)$}; & \node (N) {$H^0(A_*')$};\\
  };
  
  \draw[<<-] (M) -- (Y) node[anchor=east] {};
  \draw[->] (Y) -- (X) node[anchor=south] {};
  \draw[<<-] (N) -- (X) node[anchor=west] {};
  \draw[->] (M) -- (N) node[anchor=south] {};.
 
\end{tikzpicture} 
\end{center}induces a diagram 
\begin{center}
\begin{tikzpicture}[every node/.style={midway}]
  \matrix[column sep={16em,between origins}, row sep={4em}] at (0,0) {
    \node(Y){$\Mo_{A_*}$} ; & \node(X) {$\Mo_{A_*'}$}; \\
    \node(M) {$\Mo_{H^0(A_*)}$}; & \node (N) {$\Mo_{H^0(A_*')}$};\\
  };
  
  \draw[<<-] (M) -- (Y) node[anchor=east] {$\mathscr{H}_{A_*}$};
  \draw[->] (Y) -- (X) node[anchor=south] {$A_*'\otimes_{A_*}-$ };
  \draw[<<-] (N) -- (X) node[anchor=west] {$\mathscr{H}_{A_*'}$};
  \draw[->] (M) -- (N) node[anchor=south] {$H^0(A_*')\otimes_{H^0(A_*)}-$};.
 
\end{tikzpicture}
\end{center}
in which all the functors appearing preserve also equivalences. Hence, it descends to a diagram 
\begin{center}

\begin{tikzpicture}[every node/.style={midway}]
  \matrix[column sep={16em,between origins}, row sep={4em}] at (0,0) {
    \node(Y){$\Mod_{A_*}$} ; & \node(X) {$\Mod_{A_*'}$}; \\
    \node(M) {$\Mod_{H^0(A_*)}$}; & \node (N) {$\Mod_{H^0(A_*')}$};\\
  };
  
  \draw[<<-] (M) -- (Y) node[anchor=east] {$\mathscr{H}_{A_*}$};
  \draw[->] (Y) -- (X) node[anchor=south] {$A_*'\otimes_{A_*}-$ };
  \draw[<<-] (N) -- (X) node[anchor=west] {$\mathscr{H}_{A_*'}$};
  \draw[->] (M) -- (N) node[anchor=south] {$H^0(A_*')\otimes_{H^0(A_*)}-$};.
 \end{tikzpicture}
\end{center}of $\infty$-categories as well.

\begin{rem} Objects in $\Mod_{H^0(A_*)}$ are nothing but chain complexes of $H^0(A_*)$-modules in the usual sense. This follows from the fact that $H^0(A_*)$ is concentrated in degree $0$.
\end{rem}
\begin{defi} \label{d4.4} Let $X$ be a formal moduli problem on which a quasi-coherent sheaf  $\mathscr{F}$ is defined. Then
\begin{enumerate}
\item[(1)] $\mathscr{F}$ is said to be \emph{coherent} if for any $A_* \in \dg_k$ and any $\eta \in X(A_*)$, the image of the stalk $\mathscr{F}_{\eta}$ under the functor $\mathscr{H}_{A}$ in $\Mod_{H^0(A_*)}$, i.e.  $\mathscr{H}_{A}(\mathscr{F}_{\eta})$ has finitely generated cohomologies over $H^0(A_*)$.
\item[(2)]  $\mathscr{F}$ is said to be \emph{flat }if for any $A_* \in \dg_k$ and any $\eta \in X(A_*)$, $\mathscr{H}_{A}(\mathscr{F}_{\eta})$ is a chain complex of flat $H^0(A_*)$-modules.

\item[(3)]  $\mathscr{F}$ is said to be \emph{bounded above} if for any $A_* \in \dg_k$ and any $\eta \in X(A_*)$, $\mathscr{H}_{A}(\mathscr{F}_{\eta})$ is a bounded above chain complex of $H^0(A_*)$-modules.

\item[(4)]  $\mathscr{F}$ is said to be \emph{good} if it is bounded above and flat.
\end{enumerate}

\end{defi}
\begin{rem}\label{r4.2} If $V_*$ is a connective $\mathfrak{g}_*$-module, we say that $V_*$ is \emph{flat} if for any $A_*\in \dg_k$ and for any element $\eta\in \Map_{\Lie_k}(D(A_*),\mathfrak{g}_*)$, the $A_*$-module $C^*(D(A_*),V_*)$ (cf. Definition \ref{def2.4}), where $V_*$ is thought of as a $D(A_*)$-module via the module structure induced by $\eta$, is flat in the sense of Definition \ref{d4.4}. We say that $V_*$ is \emph{good} if it is cohomologically finite-dimensional, bounded above and flat (see also Definition \ref{d2.5}).
\end{rem}

At the present, we are in a position to give the notion of fmps with cohomological constraints and the notion of derived cohomology jump functors which in fact generalizes that of cohomology jump functors given in Definition \ref{d4.3}.
\begin{prop} \label{p4.3} Let $X$ be a formal moduli problem on which a good coherent sheaf $\mathscr{F}$ is defined. To each $A_*\in \dg_k$, the assignment  $$A_* \mapsto \lbrace \eta \in X(A_*)\mid    J_k^i(\mathscr{H}_{A_*}(\mathscr{F}_{\eta}))=0 \rbrace, $$ where $J_k^i$ is given in Definition \ref{d4.1},  gives rise to a well-defined functor $X_{\mathscr{F}}^{i,k} : \dg_k \rightarrow \ens$ .
\end{prop}
\begin{proof}
First of all, observe that for each $A_* \in \dg_k$ and each $\eta \in X(A_*)$, $\mathscr{H}_{A_*}(\mathscr{F}_{\eta})$ is a bounded-above complex of $H^0(A_*)$-modules whose cohomologies are finitely generated over $H^0(A_*)$ by the assumption that $\mathscr{F}$ is a connective good coherent sheaf. Therefore, the cohomology jump ideals  $J_k^i(\mathscr{H}_{A_*}(\mathscr{F}_{\eta}))$ are well-defined.

Given a morphism $A_*\overset{f}{\rightarrow} A_*'$ in $\dg_k$, we have a morphism of  $X(A_*)\overset{\bar{f}}{\rightarrow} X(A_*')$ of simplicial sets. Suppose that $\eta \in X^{i,k}_{\mathscr{F}}(A_*)\subset X(A_*)$ and that $\bar{f}$ sends $\eta$ to some $\eta' \in X(A_*')$. We claim that $\eta'$ is actually an element of $X^{i,k}_{\mathscr{F}}(A_*')$. Indeed,  the equivalence $$\mathscr{F}_{\eta'}\simeq A'\otimes_A \mathscr{F}_{\eta} $$ implies the following chain of equivalences
\begin{align*}
\mathscr{H}_{A_*'}(\mathscr{F}_{\eta'}) &\simeq \mathscr{H}_{A_*'}( A'\otimes_A \mathscr{F}_{\eta}  )\\
 &\simeq H^0(A_*')\otimes_{H^0(A_*)}\mathscr{H}_{A_*}(\mathscr{F}_{\eta}) 
\end{align*}
Since $\mathscr{H}_{A_*}(\mathscr{F}_{\eta})$ is a complex of flat $H^0(A_*)$-modules, by Proposition \ref{p4.1} and by the fact that cohomology jump ideals are invariant under equivalences of chain complexes, we obtain that 
\begin{align*}
J_k^i(\mathscr{H}_{A_*'}(\mathscr{F}_{\eta'}) ) &=  J_k^i( H^0(A_*')\otimes_{H^0(A_*)}\mathscr{H}_{A_*}(\mathscr{F}_{\eta}) ) \\
 &=  J_k^i( \mathscr{H}_{A_*}(\mathscr{F}_{\eta}) )  \cdot H^0(A_*').
\end{align*}
As $\eta \in X^{i,k}_{\mathscr{F}}(A_*)$, by definition, $J_k^i( \mathscr{H}_{A_*}(\mathscr{F}_{\eta}) )=0 $ and then so is $J_k^i(\mathscr{H}_{A_*'}(\mathscr{F}_{\eta'}) )$. This justifies the claim. In particular, if $A_*\overset{f}{\rightarrow} A_*'$ is an equivalence in $\dg_k$ then it induces an equivalence $X^{i,k}_{\mathscr{F}}(A_*) \rightarrow X^{i,k}_{\mathscr{F}}(A_*')$.
This finishes the well-definedness verification of $X_{\mathscr{F}}^{i,k}$.
\end{proof}

\begin{prop}\label{p4.4} Let $\mathfrak{g}_*$ be a dgLa and $V_*$ a $\mathfrak{g}_*$-module. To each $A_*\in \dg_k$, the assignment  $$A_* \mapsto \lbrace \eta \in \mathfrak{MC}_{\mathfrak{g}_*}(A_*)\mid    J_k^i(\mathscr{H}_{A_*}(V_*^{A_*,\eta}))=0 \rbrace, $$ where $J_k^i$ is given in Definition \ref{d4.1},  gives rise to a well-defined functor $\mathfrak{MC}_{\mathfrak{g}_*,V_*}^{i,k} : \dg_k \rightarrow \ens$ .

\end{prop}
\begin{proof}
Taking in account  the well-definedness of the functors given in Definition \ref{d4.1} and the isomorphism (\ref{coco}), the proof is similar to the one of Proposition \ref{p4.3} and then is left to the reader.
\end{proof}

\begin{defi} The functors $X_{\mathscr{F}}^{i,k}$ in Proposition \ref{p4.3} are called \emph{formal moduli problems with cohomological constraints associated} the pair $(X,\mathscr{F})$ while the functors $\mathfrak{MC}_{\mathfrak{g}_*,V_*}^{i,k}$ in Proposition \ref{d4.4} are called \emph{derived cohomology jump functor} associated to the pair $(\mathfrak{g}_*,V_*)$.
\end{defi}
\begin{rem} \label{r4.5} Let $X$ be a formal moduli problem. Then we can assume that $X$ is of the form $\Map_{\Lie_k}(D(-),\mathfrak{g}_*)$ where  $\mathfrak{g}_*$ is some differential graded Lie algebra and $D$ is the Koszul functor (cf. Proposition \ref{p2.2} and Theorem \ref{t2.1}). Suppose further that $V_*$ is a connective bounded-above $\mathfrak{g}_*$-module then the associated quasi-coherent $\mathscr{F}^{V_*}$ is a bounded-above good coherent sheaf on $X$. Thus, it can be easily checked that $X_{\mathscr{F}^{V_*}}^{i,k}$ and $\mathfrak{MC}_{\mathfrak{g}_*,V_*}^{i,k}$ are essentially equivalent to each other by Proposition \ref{d4.2}. In addition, the restriction of their connected components on $\Art_k$ are nothing but the classical cohomology jump functors in Definition \ref{d4.3}.\end{rem}

\section{Lurie-Pridham's equivalence with cohomological constraints} \label{s5} From now on, by $(X,\mathfrak{g}_*)$, we always mean a formal moduli problem and its associated differential graded Lie algebra. Let $\qqc(X)^{\mathcal{G}}$ and $\Mo_{\mathfrak{g}_*}^{\mathcal{G}}$ denote the full-subcategory of $\qc(X)^{cn}$ spanned by the good coherent sheaves (cf. Definition \ref{d4.4}) and the full-subcategory of $\Mod_{\mathfrak{g}_*}^{cn}$ generated by the good connective  $\mathfrak{g}_*$-modules, i.e. connective $\mathfrak{g}_*$-modules which are also cohomologically finite-dimensional, bounded above and flat (cf. Definition \ref{d2.5} and Remark \ref{r4.2}), respectively.

\begin{thm}\label{t5.1} Let $X$ be a formal moduli problem whose governing differential graded Lie algebra is $\mathfrak{g}_*$. Then there exists  an equivalence 
$$\qqc(X)^{\mathcal{G}} \overset{\simeq}{\rightarrow} \Mo_{\mathfrak{g}_*}^{\mathcal{G}}.$$ 
\end{thm}
\begin{proof} By Proposition \ref{p3.3}, we have an equivalence
$$\theta: \qc(X)^{cn} \overset{\simeq}{\rightarrow} \Mod_{\mathfrak{g}_*}^{cn}$$ of $\infty$-categories. Any morphism $\phi: \mathfrak{h}_*\rightarrow \mathfrak{g}_* $ of differential graded Lie algebras induces a morphism of formal moduli problems $Y \rightarrow X$ and then gives rise to the following diagram \begin{center}
\begin{tikzpicture}[every node/.style={midway}]
  \matrix[column sep={12em,between origins}, row sep={4em}] at (0,0) {
    \node(Y){$\qc(X)^{cn}$} ; & \node(X) {$ \Mod_{\mathfrak{g}_*}^{cn}$}; \\
    \node(M) {$\qc(Y)^{cn}$}; & \node (N) {$ \Mod_{\mathfrak{h}_*}^{cn}$};\\
  };
  
  \draw[<-] (M) -- (Y) node[anchor=east] {};
  \draw[->] (Y) -- (X) node[anchor=south] {};
  \draw[<-] (N) -- (X) node[anchor=west] {};
  \draw[->] (M) -- (N) node[anchor=south] {};.
 
\end{tikzpicture} 
\end{center}
which commutes up to homotopy. In particular, by taking $\mathfrak{h}_*=0$, it follows that the composite functor 
$$\qc(X)^{cn} \overset{\theta}{\rightarrow} \Mod_{\mathfrak{g}_*}^{cn} \rightarrow \Mod_k^{cn}$$ is nothing but the evaluation at the base point $\eta_0 \in X(k)$. This means among the other things that $\theta$ maps $\qqc(X)^{\mathcal{G}}$ into the fiber product $$\Mo_{\mathfrak{g}_*}^{\mathcal{G}}=\Mod_{\mathfrak{g}_*}^{cn} \times_{\Mod_k^{cn}} \Mo_k^{\mathcal{G}} .$$
It remains to prove that if $V_* \in \Mo_{\mathfrak{g}_*}^{\mathcal{G}}$, there exists an object $\mathscr{F}\in \qqc(X)^{\mathcal{G}}$ such that $\theta(\mathscr{F})\simeq V_*$. For this to be done, it is sufficient to take the quasi-coherent sheaf $\mathscr{F}^{V_*}$ associated to the $\mathfrak{g}_*$-module $V_*$. Indeed, by construction $\theta(\mathscr{F}^{V_*})\simeq V_*$. We shall prove that $\mathscr{F}^{V_*}\in \qqc(X)^{\mathcal{G}}$. For any $A_* \in \dg_k$, and $\eta\in X(A_*) $, let $\eta'$ be the image of $\eta$ under the morphism $X(A_*)\rightarrow X(H^{0}(A_*))$ induced by the canonical one  $A_* \twoheadrightarrow H^{0}(A_*)$. On one hand, we have the following equivalence  
\begin{align*}
 \mathscr{H}_{A_*}(\mathscr{F}^{V_*}_{\eta})&\simeq H^0(A_*)\otimes_{A_*} \mathscr{F}^{V_*}_{\eta} \\
 &\simeq \mathscr{F}^{V_*}_{\eta'}.
\end{align*} On the other hand, $\eta' \in  X(H^{0}(A_*))$ and $ H^{0}(A_*) \in \Art_k$. As a matter of course, an easy application of  Proposition \ref{p4.2} gives $$\mathscr{F}^{V_*}_{\eta'} \overset{\simeq}{\rightarrow} (V_*\otimes _{k}H^{0}(A_*),d_{\eta'}).$$ Combining with the previous equivalence, we obtain that $$\mathscr{H}_{A_*}(\mathscr{F}^{V_*}_{\eta}) \overset{\simeq}{\rightarrow} (V_*\otimes _{k}H^{0}(A_*),d_{\eta'}) $$
where the latter is well known to be a bounded above chain complex of flat $H^{0}(A_*)$-modules whose cohomologies are finitely generated over $H^{0}(A_*)$. This justifies that $\mathscr{F}^{V_*}\in \qqc(X)^{\mathcal{G}}$ and then finishes the proof.
\end{proof}




Let $\Ho(\fun(\dg_k,\ens))$ be the homotopy category of functors from $\dg_k$ to $\ens$. Then for any $i$ and $k$, Proposition \ref{p4.3} and Proposition \ref{p4.4} give rise to the following functors
\begin{align*}
X_{-}^{i,k} : \qqc(X)^{\mathcal{G}} &\rightarrow \Ho(\fun(\dg_k,\ens)) \\
\mathscr{F} &\mapsto X_{\mathscr{F}}^{ik}
\end{align*} and 
\begin{align*}
\mathfrak{MC}_{\mathfrak{g}_*,-}^{i,k} : \Mo_{\mathfrak{g}_*}^{\mathcal{G}} &\rightarrow \Ho(\fun(\dg_k,\ens))\\
\mathscr{F} &\mapsto \mathfrak{MC}_{\mathfrak{g}_*,V_*}^{i,k},
\end{align*} which then descend to functors of homotopy category
\begin{align*}
X_{-}^{i,k}: \Ho(\qqc(X)^{\mathcal{G}}) \rightarrow \Ho(\fun(\dg_k,\ens)) 
\end{align*} and 
\begin{align*}
\mathfrak{MC}_{\mathfrak{g}_*,-}^{i,k} : \Ho(\Mo_{\mathfrak{g}_*}^{\mathcal{G}}) \rightarrow \Ho(\fun(\dg_k,\ens)),
\end{align*} respectively.
\begin{coro} \label{t5.1} 
The following diagram 
$$\xymatrix{
\Ho(\Mo_{\mathfrak{g}_*}^{\mathcal{G}})  \ar[rr]^-{\mathrm{sheafification}}_-{\simeq} \ar[ddr]_-{\mathfrak{MC}_{\mathfrak{g}_*,-}^{i,k}} & &\Ho(\qqc(X)^{\mathcal{G}})\ar[ddl]^-{X_{-}^{i,k}}
\\ &   &
\\ & \Ho(\fun(\dg_k,\ens))&
}$$
commutes. In other words, derived cohomology jump functors associated to good modules over $\mathfrak{g}_*$ are homotopically equivalent to fmps with cohomological constraints associated to good coherent sheaves over $X$.

\end{coro}

\begin{proof}
This is just an application of Theorem \ref{t5.1} and what we have done so far. Given a good coherent sheaf $\mathscr{F}$, by Proposition \ref{p4.3}, formal moduli problems with cohomological constraints associated to the pair $(X, \mathscr{F})$ are uniquely determined up to equivalences. Moreover, for the sake of the equivalence  $$\qqc(X)^{\mathcal{G}} \overset{\simeq}{\rightarrow} \Mo_{\mathfrak{g}_*}^{\mathcal{G}}$$ in Theorem \ref{t5.1}, each good coherent sheaf $\mathscr{F}$ corresponds also to a good $\mathfrak{g}_*$-module $V_*$ whose underlying chain complex can be taken to be the stalk of $\mathscr{F}$ at the base point of the given formal moduli problem $X$. Applying the sheafification to this $\mathfrak{g}_*$-module $V_*$, we obtain a sheaf $\mathscr{F}^{V*}$ which is equivalent to $\mathscr{F}$. The commutativity of the diagram follows directly from Remark \ref{r4.5}. This finishes the verification.
\end{proof}
\begin{remark} We have shown that deformation problems with cohomology constraints can be naturally treated in the $\infty$-setting. At the same time, many situations are already effectively handled by identifying the associated $(\mathrm{dgLa}, \text{module})$ pair and working with the classical Maurer--Cartan equations (cf. \cite{1}). Furthermore, a more flexible framework in terms of $L_{\infty}$-pairs has been developed in \cite{BR18} together with a short guide to applications given in \cite{BD25}, which provides additional advantages, as illustrated by the examples therein (see also \cite{Bud25}).\end{remark}

\section{Some classes of fmps with cohomological constraints} \label{s6}
\subsection{Prorepresentable formal moduli problems} We recall that a formal moduli problem $X$ is said to be pro-representable if there exists a pro-object $A_* \in \dg_k$ such that $X $ is equivalent to $\Map_{\dg_k}(A_*,-)$. In this case, the associated governing differential graded Lie algebra is nothing but $D(A_*)$, where $D$ is the Koszul duality. One interesting characteristic of $D(A_*)$ is that it is cohomologically concentrated in strictly positive degrees. Surprisingly, this turns out to be also the sufficient condition for the prorepresentability of formal moduli problems: that is, a formal moduli problem whose associated differential graded Lie algebra is cohomologically concentrated in strictly positives is prorepresentable (cf. \cite[Corollary 2.3.6]{6}). Of course, the propresentability of $X$ implies that of its corresponding functor of artinian rings. More precisely, if $X$ is proreprentable by $A_*$ then the restriction of its connected components $\pi_0(X)$ on the sub-category $\Art_k$ of local artinian $k$-algebras with residue field $k$ is prorepresentable by $H^0(A_*)$ in the classical sense (cf. Remark \ref{r2.2}).

Now, let $X$ be a prorepresentable formal moduli problem, i.e. $X\simeq\Map_{\dg_k}(A_*,-) $ for some pro-object $A_* \in \dg_k$. For simplicity, we can assume that $A_*$ is actually a differential graded local artinian $k$-algebra augmented to $k$. Then the associated classical functor of artinian rings $\pi_0(X)$ is representable by $\pi_0(\Sp(A_*))$. Moreover, the $\infty$-category $\qc(X)$ of quasi-coherent sheaves has a quite simple  
description: it is equivalent to the $\infty$-category $\Mod_{A_*}$ of $A_*$-modules (cf. Remark \ref{r3.1}). In particular, let $\mathscr{F}$ be a good coherent sheaf on $X$. Then $\mathscr{F}$ is nothing but an $A_*$-module $V_*$ which is also coherent, flat and bounded above. In a fancier way, $\mathscr{F}$ can be thought of as a chain complex of $\mathcal{O}_{\Sp(A_*)}$ -modules on the derived scheme $\Sp(A_*)$. We consider the derived cohomology jump functor \begin{align*}
X_{\mathscr{F}}^{i,k} : \dg_k &\rightarrow \ens \\
B_* &\mapsto \lbrace \eta \in X(B_*)\mid    J_k^i(\mathscr{H}_{B_*}(\mathscr{F}_{\eta}))=0 \rbrace 
\end{align*} associated to the pair $(X,\mathscr{F})$.
The cohomology jump ideal of $\mathscr{F}$ generates naturally an ideal sheaf which actually defines a closed subscheme $\pi_0(\Sp(A_*))^{i,k}$ of $\pi_0(\Sp(A_*))$. Therefore, the classical cohomology jump functor $\pi_0(X_{\mathscr{F}}^{i,k})$ is represented by $\pi_0(\Sp(A_*))^{i,k}$.

\subsection{Semi-prorepresentable formal moduli problems}  It  happens quite often that formal moduli problems are not (pro)representable due to that fact that their associated differential graded Lie algebra has some non-positive components. The typical example is the formal moduli problem associated to a complex compact manifold or a projective scheme. In fact, it can be proved that the $0^{\text{th}}$-cohomology of its associated differential graded Lie algebra is the space of its vector fields which is not vanishing in general. This motivates the author in \cite{3} to introduce the notion of semi-prorepresentability. More specifically, a formal moduli problem $X$ is said to be semi-prorepresentable if there exits a pro-object $A_*$ in $\dg_k$ and an étale morphism $\Sp(A_*) \rightarrow X$ (cf. \cite[Definition 2.2]{3}). It can be checked without any difficulty that if $X$ is a semi-prorepresentable formal moduli problem then its associated functor of artinian rings admits a semi-universal element whose base is nothing but the  ``Kuranishi space'' in the classical sense. 

Now, let $X$ be a semi-prorepresentable formal moduli problem, i.e. there exists an étale morphism $\pi:\Sp(A_*) \rightarrow X$ for some pro-object $A_*$. As before, for simplicity, we suppose further that $A_*$ is really an object of $\dg_k$. Given a good coherent sheaf $\mathscr{F}$ on $X$. The pullback of $\mathscr{F}$ by $\pi$ is a good coherent sheaf $\mathscr{F}_{\pi}$ on $\Sp(A_*)$. The latter one is nothing but a coherent, flat and bounded above chain complex of $A_*$-modules. Unwinding the definition of semi-prorepresentability, $\mathscr{F}_{\pi}$ is ``semi-universal'' in the following sense: for any $B_* \in \dg_k$ and for any element $\eta: \Sp(B_*) \rightarrow X$, there exists a morphism $f_\eta: \Sp(B_*) \rightarrow \Sp(A_*)$ (not unique up to equivalences in general) such that $\pi\circ f_{\eta}= \eta$ and $$\mathscr{F}_{\eta}\simeq B_* \otimes_{A_*}^{f_{\pi}}\mathscr{F}_{\pi}. $$
Just like the pro-representable case, the jump cohomology ideal $J_k^i(\mathscr{H}_{A_*}(\mathscr{F}_{\pi}))$ defines a closed subscheme $\pi_0(\Sp(A_*))^{i,k}$ of $\pi_0(\Sp(A_*))$, by which the classical cohomology jump functor $\pi_0(X_{\mathscr{F}}^{i,k})$ is ``semi-represented".



\end{document}